\documentclass[a4paper,reqno,oneside,11pt]{amsart}
\usepackage{amsfonts,amsmath,amsthm,mathrsfs,dsfont}
\usepackage{hyperref} 
\usepackage{colordvi}
\usepackage[usenames,dvipsnames]{color}

 \numberwithin{equation}{section}
\allowdisplaybreaks[3]

\usepackage{cite}

\renewcommand{\Tilde}{\widetilde}

\newcommand{\R}{\mathbb{R}}

\newcommand{\N}{\mathbb{N}}

\newcommand{\B}{\mathscr{B}}
\newcommand{\HH}{\mathscr{H}}

\newcommand{\id}{\mathds{1}}

\newcommand{\eps}{\varepsilon}
\newcommand{\cO}{\mathcal{O}}
\newcommand{\cD}{\mathcal{D}}
\newcommand{\cH}{\mathcal{H}}

\newcommand{\cN}{\mathcal{N}}

\newcommand{\U}{\mathcal{U}}
\newcommand{\V}{\mathcal{V}}

\newtheorem{theorem}{Theorem}[section]
\newtheorem{prop}[theorem]{Proposition}
\newtheorem{cor}[theorem]{Corollary}
\newtheorem{lemma}[theorem]{Lemma}

\newtheorem{assumption}[theorem]{Assumption}

\theoremstyle{definition}

\theoremstyle{remark}
\newtheorem{rem}[theorem]{\bf Remark}

\begin{document}

\title[]{Robin eigenvalues on domains with peaks}

\author{Hynek Kova\v r\'\i k}

\address{DICATAM, Sezione di Matematica, Universit\`a degli studi di Brescia,
Via Branze 38, 25123 Brescia, Italy}
\email{hynek.kovarik@unibs.it}
\urladdr{http://hynek-kovarik.unibs.it}

\author{Konstantin Pankrashkin}

\address{Laboratoire de Math\'ematiques d'Orsay, Univ. Paris-Sud, CNRS, Universit\'e Paris-Saclay, 91405 Orsay, France}
\email{konstantin.pankrashkin@math.u-psud.fr}
\urladdr{http://www.math.u-psud.fr/~pankrashkin/}

\begin{abstract}
Let $\Omega\subset\mathbb{R}^N$, $N\ge 2,$ be a bounded domain with an outward power-like peak which is assumed not too sharp in a suitable sense.
We consider the Laplacian $u\mapsto -\Delta u$ in $\Omega$ with the Robin boundary condition
$\partial_n u=\alpha u$ on $\partial\Omega$ with $\partial_n$ being the outward normal derivative and $\alpha>0$ being a parameter.
We show that for large $\alpha$ the associated eigenvalues $E_j(\alpha)$
behave as $E_j(\alpha)\sim -\epsilon_j \alpha^\nu$, where
$\nu>2$ and $\epsilon_j>0$ depend on the dimension and the peak geometry.
This is in contrast with the well-known estimate
$E_j(\alpha)=O(\alpha^2)$ for the Lipschitz domains.
 \end{abstract}


\maketitle


\section{\bf Introduction} 

Given a domain $\Omega\subset\R^N,  N\geq 2,$ with a suitably regular boundary and a parameter $\alpha>0$, we 
consider the self-adjoint operator $Q^\Omega_\alpha$ in $L^2(\Omega)$ generated by the quadratic form
\[
q^\Omega_\alpha(u,u)=\int_\Omega |\nabla u|^2\, dx -\alpha \int_{\partial\Omega} u^2 d\sigma, \quad u\in H^1(\Omega),
\]
where $d\sigma$ stands for the $(N-1)$-dimensional Hausdorff measure. Informally, the operator $Q^\Omega_\alpha$
can be viewed as the Laplacian with the Robin boundary condition $\partial_n u=\alpha u$, where $D_n$ is the outward normal derivative.
Various properties of the operator $Q^\Omega_\alpha$
have been analysed in the literature over the last decades, see e.g. the recent paper \cite{DFK} for a review of results and a collection of open problems.
In the present paper we are interested in the asymptotic behaviour of the lowest eigenvalues $E_j(Q^\Omega_\alpha)$ of $Q^\Omega_\alpha$
in the limit $\alpha\to +\infty$. It is a standard result that for Lipschitz domains for large $\alpha$ one has the bound
$E_1(Q^\Omega_\alpha)\ge -K \alpha^2$ with $K>0$, see Subsection~\ref{ss-lip} below. Under additional regularity assumptions, i.e.
for the case when $\Omega$ is a so-called corner domain, the lower bound can be upgraded to the asymptotics $E_1(Q^\Omega_\alpha)\sim-K\alpha^2$ with some $K\ge 1$
depending on the regularity of the boundary as described in \cite{bp,lp}, in particular, $K=1$ for $C^1$ domains, see \cite{luzhu},
and for planar domains the value $K$ is determined by the smallest corner at the boundary, see \cite{kh,lp}.
More precise asymptotic expansions of $E_1(Q^\Omega_\alpha)$ have been obtained, under various geometric assumptions,
in \cite{em,p13,HK,kkr,pp15,pp15b}.  Analogous version of the problem for the $p$-Laplacian was treated in \cite{kop} by the authors.

One arrives then at the following natural question: what kind of results can be obtained for non-Lipschitz domains?
An easy revision of the above mentioned works shows that inward peaks do not influence the eigenvalue asymptotics
(the first terms in the asymptotic expansions are determined by the rest of the boundary),
so in the present work we are studying the operator $Q^\Omega_\alpha$ for domains $\Omega$
with a suitably defined outward peak. It is well-known that if the peak is too sharp, then the quadratic form  $q^\Omega_\alpha$
fails to be semibounded, hence, the first eigenvalue of $Q^\Omega_\alpha$ does not exist, see Remark~\ref{rem-1} below, 
so our objective is to describe the asymptotic behavior of the eigenvalues in a more detailed way for the peak of a ``moderate'' sharpness.
More precisely, we restrict our attention to power-like peaks characterized by two parameters
\[
p\in(1,2), \quad m>0
\]
as follows:
\begin{assumption}\label{def-1peak}
There exist $\ell_0>0$ and a positive $C^1$ function $\mu$ on $(-\ell_0,\ell_0)^{N-1}$ with $\mu(0)=m$ such that
\begin{itemize}
\item for some $\ell\in (0,\ell_0)$ one has
\[
\Omega\cap (-\ell,\ell)^N=\Big\{ (x_1,x') \in \R\times \R^{N-1}: \, x_1\in (0,\ell),\  \dfrac{|x'|}{\mu(x')}< x_1^p\Big\},
\]
\item for some $h\in (0,\ell)$ the domain $\Omega\setminus [-h,h]^N$ is Lipschitz.
\end{itemize}
\end{assumption}
In order to state the main results we need some notation and an auxiliary one-dimensional operator.
It will be convenient to use the shorthand
\[
n:=N-1.
\]
If $H$ is a self-adjoint operator, we denote by $E_j(H)$ its $j$th eigenvalue (when enumerated in the non-decreasing order and counted according to the multiplicities) if it exists.
If $E\in\R$, we denote by  $\cN(H,E)$ the number of eigenvalues of $H$ in $(-\infty,E)$. 

For $\lambda>0$, consider the symmetric differential operator in $L^2(0,\infty)$ given by
\[
C_0^\infty(0,\infty) \ni f \mapsto -f'' + \bigg(\frac{(np-1)^2-1}{4s^2}-\frac{n}{\lambda s^p}\bigg) f
\]
and denote by $A_{\lambda}$ its Friedrichs extension in $L^2(0,\infty)$, then it is standard to see that
the essential spectrum of $A_{\lambda}$ is $[0,+\infty)$ and that $A_{\lambda}$ has infinitely many negative eigenvalues $E_j(A_\lambda)$ accumulating at zero
(see Subsection~\ref{compop} for a more detailed discussion).
Our result on the asymptotics of individual eigenvalues of $Q^\Omega_\alpha$ is as follows:

\begin{theorem} \label{thm-main1}
For any fixed $j\in\N$ one has
\[
E_j(Q^\Omega_\alpha) =\, \Big(\dfrac \alpha m \Big)^{\frac{2}{2-p}}\, E_j(A_1) + o(\alpha^{\frac{2}{2-p}})
\]
as $\alpha$ tends to $+\infty$.
\end{theorem}
In addition, we provide an estimate for the number of eigenvalues below a moving threshold:
\begin{theorem} \label{thm-main2}
For some $B>0$ there holds
\[
\cN(Q^\Omega_\alpha,-B\alpha^{p+1})= M_p\, \alpha^{\frac{p-1}{2}} + o(\alpha^{\frac{p-1}{2}})
\]
as $\alpha$ tends to $+\infty$, where
\[
M_p=\dfrac{1}{2\pi} \, B^{\frac{p-2}{2p}} \Big(\dfrac{n}{m}\Big)^{\frac 1 p}\int_0^1\sqrt{\dfrac{1-s^p}{s^p}}\, ds.
\]
\end{theorem}

\begin{rem} \label{rem-1}
Let us explain the restriction $p\in (1,2)$ in Assumption \ref{def-1peak}. If $p\leq 1$, then $\Omega$ is Lipschitz. Hence 
the condition $p>1$ is necessary for $\Omega$ to have a peak. On the other hand, for $p>2$ one has $E_1(Q^\Omega_\alpha) = -\infty$
for all $\alpha>0$, and for $p=2$ one has $E_1(Q^\Omega_\alpha) > -\infty$ only if $\alpha \leq \alpha_0$ with some $\alpha_0>0$, see 
\cite{dan,mp,nt}.  As explained in \cite{dan}, this is equivalent to the fact that for $p>2$ the trace operator $H^1(\Omega) 
\mapsto L^2(\partial\Omega)$ does not exist, and for $p=2$ it exists, but is not compact. 
\end{rem}

\begin{rem}
An upper bound for $E_1(Q^\Omega_\alpha)$ on the planar domain $\Omega=\big\{(x_1,x_2):|x_2|<x_1^p\big\}$
was obtained in \cite[Example~3.4]{lp} using a test function argument:
it was shown that $E_1 (Q^\Omega_\alpha)\, \leq \, -C\, \alpha^{\frac{2}{2-p}}$ for large $\alpha>0$
and some $-C>0$.  
\end{rem}

\begin{rem}
Due to the assumption $p\in(1,2)$ one has $\frac{2}{2-p}>2$ which indeed shows that the eigenvalues $E_j(Q^\Omega_\alpha)$
tend to $-\infty$ much faster then for the Lipschitz case. Furthermore, the gaps $G_j:=E_{j+1}(Q^\Omega_\alpha)-E_j(Q^\Omega_\alpha)$
has the same order in $\alpha$, which is in contrast to the previously studied cases with more regularity: as shown in \cite{kh},
for curvilinear polygons one has $G_j=O(\alpha^2)$, and for $C^k$ smooth domains one has $G_j=o(\alpha)$ if $k=2$ and $G_j=O(\sqrt{\alpha})$ if $k=3$,
see \cite{pp15b}.
\end{rem}

\begin{rem}
Concerning Theorem~\ref{thm-main2} we remark that if $\Omega$ is Lipschitz, $B>0$ and $p>1$, then $\cN(Q^\Omega_\alpha,-B\alpha^{p+1})=0$
for $\alpha$ large enough, see Corollary \ref{cor-gris}. Therefore, the growing number of eigenvalues below $-B\alpha^{p+1}$ is purely due to the presence of the peak.
\end{rem}

\section{\bf Scheme of proof}  \label{sec-prelim}

In order to prove the main results we perform first a number of truncations and dilations in order
to isolate the peak and to reudce to the problem to the study of some models domains.

\subsection{Lipschitz domains}\label{ss-lip}
Let us recall some known facts about Robin Laplacians on Lipschitz domains. The following result is quite standard,
see e.g. in \cite[Theorem~1.5.1.10]{Gris}:

\begin{prop}\label{gris0}
Let $U$ be a bounded Lipschitz domain. Then there exists a constant $K>0$ such that
\[
\eta \int_U |\nabla u|^2 dx \, +\,   \eta^{-1} \int_U u^2 dx  \  \ge  \ K  \int_{\partial U} u^2\, d\sigma
\]
holds for all $u\in H^1(U)$ and all $\eta \in (0,1)$. 
\end{prop}
An immediate consequence of Proposition \ref{gris0} is the following
\begin{cor} \label{cor-gris}
Let $U$ be a bounded Lipschitz domain. Then there exist constants $K>0$ and $\alpha_0>0$ such that 
\begin{equation} \label{lowerb-gris}
E_1(Q^U_\alpha) \, \geq \, -K\, \alpha^2 \qquad \forall\, \alpha\geq \alpha_0. 
\end{equation}
\end{cor}

\subsection{Isolating the peak}

Recall that $\Omega$ satisfies Assumption~\ref{def-1peak}, which implies a choice of strictly positive constants $\ell$ and $h$ appearing in the formulation.
For $\delta\in(0,h)$ denote
\begin{align*}
\Lambda_\delta&:= \Big\{ (x_1,x'): x_1\in (0,\delta), \  \dfrac{|x'|}{\mu(x')} < x_1^p\Big\},\\
\Omega_\delta^*&:=\Omega\setminus \overline{\Lambda_\delta},\\
\partial_0\Lambda_\delta&:=\big\{(x_1,x_2) \in \partial \Lambda_\delta: x_1<\delta\big\},
\end{align*}
then $\overline{\Omega}=\overline{\Lambda_\delta\cup \Omega_\delta^*}$ and $\Omega_\delta^*$ is a bounded Lipschitz domain by construction.
A standard application of the min-max principle
shows that for any $j\in\N$ one has
\begin{equation}
    \label{eq-base1}
E_j(R^{N,\delta}_\alpha\oplus K^{N,\delta}_{\alpha})\le E_j(Q^\Omega_\alpha)\le E_j(R^{D,\delta}_\alpha)
\end{equation}
where $R^{N/D,\delta}_\alpha$ are the self-adjoint operators in $L^2(\Lambda_\delta)$
defined respectively by the quadratic forms
\begin{align*}
r^{N,\delta}_\alpha(u,u)&=\int_{\Lambda_\delta}|\nabla u|^2dx-\alpha \int_{\partial_0\Lambda_\delta} u^2d\sigma,
\quad \cD(r^{D,\delta}_\alpha)=H^1(\Lambda_\delta),\\
r^{D,\delta}_\alpha(u,u)&=r^{N,\delta}_\alpha(u,u), \quad \cD(r^{D,\delta}_\alpha)=\big\{ u\in H^1(\Lambda_\delta):u(\delta,\cdot)=0\big\},
\end{align*}
and  $K^{N,\delta}_\alpha$ is the self-adjoint operator in $L^2(\Omega^*_\delta)$
defined by the quadratic form
\begin{align*}
k^{N,\delta}_\alpha(u,u)&=\int_{\Omega^*_\delta}|\nabla u|^2 dx-\alpha \int_{\partial\Omega^*_\delta \cap \partial\Omega} u^2\, d\sigma,
\quad \cD(k^{N,\delta}_\alpha)=H^1(\Omega^*_\delta).
\end{align*}

\begin{prop}\label{prop1}
For any $\delta\in(0,h)$ there exist constants $\alpha_0>0$ and  $K>0$ such that for $\alpha>\alpha_0$ and
$j\in\big\{1,\dots,\cN(Q^\Omega_\alpha,-K\alpha^2)\big\}$
there holds $E_j(R^{N,\delta}_\alpha)\le E_j(Q^\Omega_\alpha)\le E_j(R^{D,\delta}_\alpha)$.
\end{prop}

\begin{proof}
For large $\alpha$ for some $K>0$ one has $K^{N,\delta}_\alpha\ge -K\alpha^2$ due to Corollary \ref{cor-gris} and the min-max principle.
Hence for $E_j(Q^\Omega_\alpha)<-K\alpha^2$ one has
$E_j(R^{N,\delta}_\alpha\oplus S^{N,\delta}_{\alpha})=E_j(R^{N,\delta}_\alpha)$.
\end{proof}

\subsection{Reduction to a model peak}
In order to estimate the eigenvalues of $R^{D/N,\delta}_\alpha$ we compare them with Robin Laplacians on some model domains.
Namely, for $k>0$ and $a>0$ denote
\begin{align}
V_{k,a}&=\big\{(x_1,x')\in \R\times\R^n:\, x_1\in(0,a), \, |x' | <  k x_1^p\big\}\subset\R^N, \nonumber \\
\partial_0 V_{k,a}&=\big\{(x_1,x')\in \R\times\R^n:\, x_1\in(0,a), \, |x'|= k x_1^p \big\}\subset \partial V_{k,a},   \label{def-veps} \\
\Tilde H^1_0(V_{k,a})&=\big\{u\in H^1(V_{k,a}):\, u(a,\cdot)=0\big\}. \nonumber
\end{align}
Let $S^{k,a}_\alpha$ and $\Tilde S^{\, k,a}_\alpha$ be the self-adjoint operators in $L^2(V_{k,a})$ generated respectively
by the quadratic forms $s^{k,a}_\alpha$ and $\Tilde s\,^{k,a}_\alpha$ given by
\begin{align*}
s^{k,a}_\alpha (u,u)&=\int_{V_{k,a}} |\nabla u|^2dx-  \alpha \int_{\partial_0 V_{k,a}} u^2\, ds, \quad \cD(q_{k,a})=H^1(V_{k,a}),\\
\Tilde s^{\, k,a}_\alpha(u,u)&=s^{k,a}_\alpha(u,u), \quad \cD(\Tilde s^{\, k,a}_\alpha)=\Tilde H^1_0(V_{k,a}).
\end{align*}

\begin{prop}\label{prop2}
There exist $c>0$ and $\delta_0>0$ such that for all $\delta\in(0,\delta_0)$ and $\alpha>0$
there holds
\begin{align*}
E_j(R^{N,\delta}_\alpha)&\ge (1-c\delta) E_j(S^{m,\delta}_{\alpha_N}), \quad \alpha_N:=\dfrac{1+c\, \delta}{1-c\, \delta}\, \alpha,\\
E_j(R^{D,\delta}_\alpha)& \le (1+c\delta) E_j(\Tilde S^{m,\delta}_{\alpha_D}), \quad \alpha_D:=\dfrac{1-c\, \delta}{1+c\, \delta}\, \alpha.
\end{align*}
\end{prop}

\begin{proof}
Consider the map $\Phi:\overline{\Lambda_\delta} \to \overline{V_{m,\delta}}$ given by
\[
\Phi(x_1,x')=\Big(x_1, \dfrac{m}{\mu(x')}x' \Big),
\]
then $\nabla\Phi (x)=\text{Id}+ \mathcal{O}(|x|)$  for $x\to 0$. Hence, for sufficiently small $\delta$ the map $\Phi$ is a diffeomorphism,
and its inverse $\Psi: \overline{V_{m,\delta}}\to \overline{\Lambda_\delta}$ satisfies $\nabla\Psi (x)=\text{Id}+ \mathcal{O}(|x|)$ as well.
Therefore, there exist $c_0>0$ and $\delta_0>0$ such that that for $\delta\in(0,\delta_0)$
and all $u\in H^1(\Lambda_\delta)$ one can estimate, with $v:=u\circ \Psi\in H^1(V_{m,\delta})$,
\begin{eqnarray*}
(1-c_0\delta)\, \|v\|^2_{L^2(V_{m,\delta})} \le & \|u\|^2_{L^2(\Lambda_\delta)}&\le (1+c_0\delta)\, \|v\|^2_{L^2(V_{m,\delta})},\\
(1-c_0\delta) \displaystyle\int_{V_{m,\delta}} |\nabla v|^2\,dx  \le & \displaystyle\int_{\Lambda_\delta} |\nabla u|^2\,dx & \le (1+c_0\delta) \displaystyle\int_{V_{m,\delta}} |\nabla v|^2\,dx,\\
(1-c_0\delta) \displaystyle\int_{\partial_0 V_{m,\delta}} v^2\,ds \le & \displaystyle\int_{\partial_0\Lambda_\delta} u^2\,ds & \le (1+c_0\delta) \displaystyle\int_{\partial_0 V_{m,\delta}} v^2\,ds.
\end{eqnarray*}
The substitution of these inequalities into the expressions for $s^{k,a}_\alpha$ and $\Tilde s^{\, k,a}_\alpha$ 
and into the min-max principle gives the result.
\end{proof}

\subsection{The rescaled peak} 
\noindent Now  we need to study the eigenvalues of $S^{m,\delta}_\alpha$ and $\Tilde S^{\, m,\delta}_\alpha$ for large $\alpha$.
It is easy to see that
\[
\alpha V_{m,\delta}=  V_{m\alpha^{1-p}, \delta\alpha},
\]
and the change of variables $x=\frac y\alpha$ in the above expressions leads to the equalities
\begin{equation}
 \label{eq-base2}
E_j\big (S^{m,\delta}_\alpha\big)=\alpha^2 E_j\big(S^{m\alpha^{1-p}, \delta\alpha}_1\big),
\quad
E_j\big(\Tilde S^{\, m,\delta}_\alpha\big)=\alpha^2 E_j\big(\Tilde S^{\, m\alpha^{1-p}, \delta\alpha}_1\big).
\end{equation}
Hence we denote
\begin{equation} \label{rescaling}
\eps:=m\alpha^{1-p}\, , \qquad \text{so that } \ \delta\alpha=b \eps^{-\frac{1}{p-1}}, \quad b:=m^{\frac{1}{p-1}}\, \delta
\end{equation}
and study the rescaled operators
\begin{equation} \label{Q-eps}
Q_{\eps,b}:=S^{\, \eps, \, b \, \eps^{\frac{1}{1-p}}}_1, \quad \Tilde Q_{\eps,b}:=\Tilde S^{\, \eps\, ,b\,  \eps^{\frac{1}{1-p}}}_1
\end{equation}
as $\eps\to 0$. In Section~\ref{sec-qeb} we prove the following crucial result: 

\begin{prop} \label{prop-main}
There exist $K_1>0$, $k_1>0$, $\eps_0>0$ such that
\begin{equation} \label{main-upperb}
E_j(\Tilde Q_{\eps,b})\le (1-k_1\eps)\, \eps^{\frac{2}{p-2}}E_j(A_1) \text{ for all } \eps\in(0,\eps_0), \ 1\leq j \leq \cN(A_1,-K_1\eps^{\frac{p}{2-p}}). 
\end{equation}
Furthermore, one can find $K_2>0$, $k_2>0$, $B>0$ with
\begin{equation} \label{main-lowerb}
E_j(Q_{\eps,b})\ge (1+k_2\eps)\, \eps^{\frac{2}{p-2}} E_j(A_1) - K_2
\text{ for all } \eps\in(0,\eps_0),\ 1\leq j \leq \cN(Q_{\eps,b},-B/\eps).
\end{equation}
\end{prop}

 Before presenting the main results of our paper, let us state two simple but important consequences of Proposition \ref{prop-main}. 

\begin{cor}\label{prop4}
There exist $K_1>0$, $k_1>0$, $\alpha_0>0$ such that
\[
E_j(\Tilde S^{\, m,\delta}_\alpha)\le (1-k_1\alpha^{1-p})\,  \Big(\frac \alpha m \Big)^{\frac{2}{2-p}}\,  E_j(A_1)
\text{ for all }  \alpha>\alpha_0, \ 1\leq  j\leq \cN(A_1,-K_1\alpha^{\frac{p(1-p)}{2-p}}). 
\]
Furthermore, there exist  $K_2>0$, $k_2>0,$ and $B>0$ such that
\[
E_j(S^{m,\delta}_\alpha)\ge (1+k_2\alpha^{1-p})\,  \Big(\frac \alpha m \Big)^{\frac{2}{2-p}}\,  E_j(A_1) - K_2
\text{ for all }  \alpha>\alpha_0, \  1\leq j\leq \cN(S^{m,\delta}_\alpha,-B \alpha^{p+1}).
\]
\end{cor}

\begin{proof}
The inverse passage from $\eps$ to $\alpha$ in Proposition \ref{prop-main} implies the claim.
\end{proof}

\begin{cor}\label{prop5}
There exists $\delta_0>0$ such that for any $\delta\in(0,\delta_0)$ the following assertions hold true:

\begin{enumerate}
\item[(a)] there exist $K_1>0$, $k_1>0$, $\alpha_0>0$ such that
\[
E_j(R^{D,\delta}_\alpha)\le (1-k_1 \delta^{p-1}) \,  \Big(\frac \alpha m \Big)^{\frac{2}{2-p}}\,  E_j(A_1)
\text{ for all }  \alpha>\alpha_0, \ 1\leq j\leq \cN(A_1,-K_1\alpha^{\frac{p(1-p)}{2-p}}),
\]
\item[(b)]
there exist  $K_2>0$, $k_2>0$, $B>0$ such that
\[
E_j(R^{N,\delta}_\alpha)\ge (1+k_2\delta^{p-1}) \,  \Big(\frac \alpha m \Big)^{\frac{2}{2-p}}\,  E_j(A_1) - K_2
\text{ for all }  \alpha>\alpha_0, \  1\leq   j\leq \cN(R^{N,\delta}_\alpha,-B \alpha^{p+1}).
\]
\end{enumerate}
\end{cor}

\begin{proof}
This follows by substituting Corollary~\ref{prop4} into Proposition~\ref{prop2}. 
\end{proof}

\medskip


\subsection{\bf Proof of main results}

In order to prove Theorem~\ref{thm-main1} it suffices to insert the inequalities of Corollary~\ref{prop5} into Proposition~\ref{prop1} and to remark that $\delta>0$ can be chosen arbitrarily small.

Let us now turn to a proof of Theorem~\ref{thm-main2}. To this aim, we need an additional result on the operator $A_1$, which is proved in subsection~\ref{compop}.
\begin{prop} \label{prop-weyl}
For $\eps\to 0^+$ there holds $\cN(A_1,-\eps) =  J_p \, \eps^{-\frac{2-p}{2p}} + o\big(\eps^{-\frac{2-p}{2p}}\big)$ with
\[
J_p  = \dfrac{n^{\frac 1p}}{2\pi}\int_0^1\sqrt{\dfrac{1-s^p}{s^p}}\, ds.
\]
\end{prop}

Now let us take $\delta_0>0$ sufficiently small and  $\delta\in(0,\delta_0)$.
Let $B>0$ and $\alpha>0$ be sufficiently large. Using the lower bound of Proposition~\ref{prop1} we see that if $E_j(Q^\Omega_\alpha)\le -B\alpha^{p+1}$, then
one also has $E_j(R^{N,\delta}_\alpha)<-B\alpha^{p+1}$. By Corollary~\ref{prop5}, if $B$ is suitable chosen, then with 
\[
E_j(A_1)\le -(1-k \delta^{p-1}) B m^{\frac{2}{2-p}} \alpha^{\frac{p(p-1)}{p-2}},
\]
where $k$ does not depend on $\delta$.
It follows that
\[
\cN(Q^\Omega_\alpha,-B\alpha^{p+1})\le \cN\Big(A_1, -(1-k \delta^{p-1}) B m^{\frac{2}{2-p}} \alpha^{\frac{p(p-1)}{p-2}}\Big),
\]
and by Proposition~\eqref{prop-weyl} one has
\[
\cN(Q^\Omega_\alpha,-B\alpha^{p+1})\le J_p (1-k \delta^{p-1})^{-\frac{2-p}{2p}} B^{-\frac{2-p}{2p}} m^{-\frac{1}{p}} \alpha^{\frac{p-1}{2}} + o(\alpha^{\frac{p-1}{2}}).
\]
As $\delta>0$ can be chosen sufficiently small, we obtain the upper bound of Theorem~\ref{thm-main2}. The lower bound is obtained in an analogous way.


\subsection{Outline of the paper} 


The rest of the paper is dedicated to the proof of the key Proposition \ref{prop-main}. This is done in several steps. First we prove some auxiliary results on the one-dimensional operator $A_1$ and show that it can be approximated by a truncated operator $L_{\eps,a}$ acting on an interval $(0,a)$ for small $\eps$, see Section \ref{sec-aux}. In Section \ref{sec-1d} we study the effective contribution from the Robin Laplacian defined on the cross-section of the peak, see Lemma \ref{lem-1}. In Section \ref{sec-teps} we establish a connection between the 
truncated one-dimensional operator $L_{\eps,a}$ and the operator $\Tilde S_1^{\eps,a}$ for a fixed value of $a$ and small $\eps$. Finally, the operators $Q_{\eps,b}$ and $\Tilde Q_{\, \eps,b}$ are studied in Section \ref{sec-qeb} using an additional truncation, which completes the proof.


\section{\bf Auxiliary estimates}
\label{sec-aux}

In this section we prove a number of estimates for various operators appearing in the proof of Proposition~\ref{prop-main}.
For a scalar product in a Hilbert space $\HH$
we will use the symbol $\langle \cdot\, , \cdot \rangle_{\HH}$. 
Given $r>0$, we denote by 
$$
\B_r = \{x\in \R^n:\, |x| < r \} 
$$ 
the $n$-dimensional ball of radius $r$ entered in the origin.
Finally, $\omega_k$ stands for the surface area of the $k$-dimensional unit sphere.

\subsection{One-dimensional comparison operators}\label{compop}

For $\lambda>0$, consider the symmetric differential operator in $L^2(0,\infty)$ given by
\[
C_0^\infty(0,\infty) \ni f \mapsto -f'' + \bigg(\frac{(np-1)^2-1}{4s^2}-\frac{n}{\lambda s^p}\bigg) f.
\]
Since $(np-1)^2>0$, the operator is semi-bounded from below in view of the classical Hardy inequality,
\begin{equation} \label{hardy}
\int_0^\infty |f'|^2 ds\ \ge\  \int_0^\infty \dfrac{f^2}{4s^2}\ ds  \text{ for }  f\in C^\infty_0(0,\infty),
\end{equation}
and we denote by $A_{\lambda}$ its Friedrichs extension in $L^2(0,\infty)$.
The potential term is for large enough $s$ attractive and decays at infinity as $s^{-p}$. Hence standard spectral theory arguments show
that the essential spectrum of $A_{\lambda}$ is $[0,+\infty)$ and that $A_{\lambda}$ has infinitely many negative eigenvalues
accumulating at zero. Moreover, a scaling argument shows the equalities
\begin{equation}
   \label{eq-hom0}
E_j( A_{\kappa\lambda}) = \kappa^{-\frac{2}{2-p}}\, E_j( A_{\lambda})\, , \quad \kappa>0,
\end{equation}
in particular, the individual eigenvalues of $A_\lambda$ are continuous in $\lambda$.

In what follows we will work with truncated versions of $A_\lambda$. Namely,
given $\lambda>0$ and $a>0$ we denote by $L_{\lambda,a}$ and $M_{\lambda,a}$ the Friedrichs extensions in $L^2(0,a)$ 
and $L^2(a,\infty)$ of the operators $C_0^\infty(0,a)\ni f\mapsto A_\lambda f$ and $C_0^\infty(a,\infty)\ni f\mapsto A_\lambda f$
respectively.

\begin{proof}[\bf Proof of Proposition \ref{prop-weyl}]
Since imposing a Dirichlet boundary at one point represents a perturbation of rank one of the resolvent of $A_1$, it follows that 
\[
\cN(A_1,-\eps)\, \leq \cN(L_{1,a}  \oplus M_{1,a}, -\eps)  +1
\]
holds for all $a>0$. As the operator $L_{1,a}$ has compact resolvent, it follows that $N_a:=\cN(L_{1,a},0)<+\infty$.
Hence the above inequality and the min-max principle show that 
\begin{equation} \label{sandw}
\cN(M_{1,a}, -\eps)  \, \leq\, \cN(A_1,-\eps)\, \leq \cN(M_{1,a}, -\eps)  +N_a+1.
\end{equation}
Let $\delta>0$, then the parameter $a$ can be chosen sufficiently large to have
\[
\frac{(np-1)^2-1}{4s^2} \ge -\dfrac{\delta}{s^p} \text{ for } s\in(a,\infty).
\]
Hence, if denote by $K_{k,a}$ the self-adjoint operator in $L^2(a,\infty)$ obtain as the Friedrichs extension of
\[
C_0^\infty(a,\infty)\ni f\mapsto -f'' -\frac{k}{s^p} \, f, \quad k>0,
\]
one has the form inequality $K_{n+\delta,a}\le M_{1,a}\le K_{n,a}$ implying, for any $\eps>0$,
\begin{equation} \label{sandw2}
\cN(K_{n,a}, -\eps)  \, \leq\, \cN(A_1,-\eps)\, \leq \cN(K_{n+\delta,a}, -\eps)  +N_a+1.
\end{equation}

At any fixed values of $k$ and $a$, the operator $K_{k,a}$ can be analyzed using standard approaches, in particular, by \cite[Theorem~XIII.82]{RS4} we have, for $\eps\to 0+$, 
\[
\cN(K_{k,a}, -\eps) \sim  \dfrac{1}{2\pi} \int_a^\infty \sqrt{\Big(\frac{k}{s^p}-\eps\Big)_+} ds
\]
where $x_+=x $ for $x\ge 0$ and $x_+=0$ for $x<0$, and an elementary analysis shows that
\[
\cN(K_{k,a}, -\eps)\sim \dfrac{k^{\frac{1}{p}}}{2\pi} \eps^{-\frac{2-p}{2p}} \int_0^1\sqrt{\dfrac{1-s^p}{s^p}}ds.
\]
It remains to substitute the last estimate into \eqref{sandw2} and to use the fact that $\delta>0$ can be chosen arbitrarily small.
\end{proof}

We are now going to relate the eigenvalues of  $L_{\lambda,a}$ to those of the comparison operator $A_\lambda$.
First remark that due to the min max-principle one has
\begin{equation}
   \label{eq-1dd}
	E_j(L_{\lambda,a})\ge E_j(A_\lambda) \text{ for any $a>0$, $\lambda>0$, $j\in\N$.}
\end{equation}
Let us now obtain an asymptotic upper bound for $E_j(L_{\lambda,a})$.

\begin{lemma}\label{lem7}
Let $a>0$. Then there exist $K>0$, $k>0$, $\eps_0>0$ such that
\[
E_j(L_{\eps,a}) \le \eps^{-\frac{2}{2-p}}\, E_j( A_1) +k
\text{ for all } \eps\in(0,\eps_0), \ \ j\in\big\{ 1, \dots, \cN(A_1, -K\eps^{\frac{p}{2-p}})\big\}.
\]
\end{lemma}

\begin{proof}
The proof is quite standard using a so-called IMS partition of unity.
Let $\chi_1$ and $\chi_2$ be two smooth functions on $\R$ such that
$\chi_1^2+\chi_2^2=1$, $\chi_1(s)=0$ for $s>\frac{3}{4}\, a$, $\chi_2(s)=0$ for $s<\frac{1}{2}\, a$.
We set $k:=\|\chi_1'\|_\infty^2+\|\chi_2'\|_\infty^2$.
A direct computation shows that for $f\in C^\infty_0(0,\infty)$ one has
\begin{align*}
\int_0^\infty (f')^2 ds & =\int_0^\infty \big|(\chi_1 f)'\big|^2 ds +\int_0^\infty \big|(\chi_2 f)'\big|^2 ds
-\int_0^\infty \Big((\chi_1')^2+(\chi_2')^2\Big) f^2\, ds\\
& \ge \int_0^\infty \big|(\chi_1 f)'\big|^2 ds +\int_0^\infty \big|(\chi_2 f)'\big|^2 ds
- k \|f\|^2_{L^2(0,\infty)} .
\end{align*}
Hence
\[
\big\langle f, A_{\eps} f\big\rangle_{L^2(0,\infty)}
\ge \big\langle \chi_1 f, A_{\eps}(\chi_1 f)\big\rangle_{L^2(0,\infty)}+\big\langle \chi_2 f, A_{\eps}(\chi_2 f)\big\rangle_{L^2(0,\infty)}
- k \|f\|^2_{L^2(0,\infty)}, 
\]
which can be rewritten as
\[
\big\langle f, A_{\eps} f\big\rangle_{L^2(0,\infty)}+k \|f\|^2_{L^2(0,\infty)}\ge \big\langle \chi_1 f, A_{\eps}(\chi_1 f)\big\rangle_{L^2(0,a)}+\big\langle 
\chi_2 f, A_{ \eps}(\chi_2 f)\big\rangle_{L^2(a/4,\infty)}.
\]
Using the equality
\[
\|f\|^2_{L^2(0,\infty)}=\|\chi_1 f\|^2_{L^2(0,a)}+\|\chi_2 f\|^2_{L^2(a/4,\infty)},
\quad \chi_1 f\in C^\infty_0(0,a), \quad \chi_2 f\in C^\infty_0(a/4,\infty)
\]
and the min-max principle one obtains
\begin{align}
E_j(A_{\eps})+k & \ge \inf_{\substack{S\subset C^\infty_0(0,\infty)\\ \dim S=j}} \sup_{f\in S,\, f\ne 0} \dfrac{\big\langle \chi_1 f, A_{\eps}(\chi_1 f)\big\rangle_{L^2(0,a)}+\big\langle \chi_2 f, A_{\eps}(\chi_2 f)\big\rangle_{L^2(a/4,\infty)}}{\|f\|^2_{L^2(0,\infty)}} \nonumber \\
& = \inf_{\substack{S\subset C^\infty_0(0,\infty)\\\dim S=j}} \sup_{f\in S, \,f\ne 0} \dfrac{\big\langle \chi_1 f, A_{\eps}(\chi_1 f)\big\rangle_{L^2(0,a)}+\big\langle \chi_2 f, A_{\eps}(\chi_2 f)\big\rangle_{L^2(a/4,\infty)}}{\|\chi_1 f\|^2_{L^2(0,a)}+\|\chi_2 f\|^2_{L^2(a/4,\infty)}}  \nonumber \\
& \ge \inf_{\substack{S\subset C^\infty_0(0,a)\oplus C^\infty_0(a/4,\infty)\\ \dim S=j}}\,  \sup_{(u_1,u_2) \in S} \dfrac{\big\langle u_1 , A_{\eps}\, u_1\big\rangle_{L^2(0,a)}
+\big\langle u_2 , A_{\eps}\, u_2\big\rangle_{L^2(a/4,\infty)}}{\|u_1\|^2_{L^2(0,a)}+\|u_2\|^2_{L^2(a/4,\infty)}}  \nonumber \\
& =E_j\big(L_{\eps,a}  \oplus M_{\eps, a/4}\big) .
    \label{eq-ineq00}
\end{align}
With the help of the Hardy inequality \eqref{hardy} we conclude that $M_{\eps, a/4}\ge  -K_0 \eps^{-1}$ for $K_0:=4^p a^{-p}$.
Now take any $K>K_0$ and set $\eps_0:=(K-K_0)/k$, then for any $j\le \cN(A_1, -K\eps^{\frac{p}{2-p}})$ and $\eps\in(0,\eps_0)$ one has then, using  \eqref{eq-hom0},
\[
E_j(A_{\eps})+k=\eps^{-\frac{2}{2-p}} E_j(A_1)+k
< -K\eps^{-1}+k=-(K-k\eps) \eps^{-1} \le  -K_0\eps^{-1}\le E_1(M_{\eps, a/4}). 
\]
This in combination with \eqref{eq-ineq00} shows that $E_j(L_{\eps,a} \oplus M_{\eps,a/4})=E_j(L_{\eps,a})$ and the result follows. 
\end{proof}

\begin{lemma}\label{lem-nnl}
Let $a>0$, then there exist $\eps_0>0$ and $K>0$ such that
\[
\cN(L_{\eps,a},0)\ge \cN(A_{1}, -K\eps^{\frac{p}{2-p}}) \qquad \forall\, \eps\in(0,\eps_0). 
\]
\end{lemma}

\begin{proof}
It follows from Lemma~\ref{lem7} that one can find $K>0$, $k>0$ and $\eps_0>0$ such that
\[
E_j(L_{\eps,a})\le -\dfrac{K}{\eps}+k \text{ for $\eps\in(0,\eps_0)$ and $j\le \cN(A_1,-K\eps^{\frac{p}{2-p}})$.}
\]
Adjust the value $\eps_0$ to have $-K/\eps_0+k< 0$, then
$E_j(L_{\eps,a})<0$ for all $j\le \cN(A_1,-K\eps^{\frac{p}{2-p}})$.
\end{proof}

\subsection{\bf Robin Laplacian on a ball} 
\label{sec-1d}

Let $B_{\eps,r}$ be the self-adjoint operator in $L^2(\B_\eps)$ generated by the quadratic form 
\begin{equation} \label{d-form} 
b_{\eps,r}(f,f) =\int_{\B_\eps} |\nabla f|^2\, dy   -  r\! \int_{\partial\B_\eps} |f|^2\, d\sigma,  \qquad  f\in H^1(\B_\eps),
\end{equation} 
where $r\in \R$ and let $E_j(B_{\eps,r})$ denote the eigenvalues of $B_{\eps,r}$.
In other words, the operator $B_{\eps,r}$ is the Laplacian $f\mapsto -\Delta f$
with the Robin boundary condition $D_n f = r f$ with $D_n$ being the outward normal derivative.

\smallskip

\noindent We have

\begin{lemma} \label{lem-1}

The following assertions hold true:

\begin{enumerate}

\item[(a)] $E_j(B_{\eps,r}) = \eps^{-2} E_j(B_{1,\eps r})$. 

\smallskip

\item[(b)] The mapping $\R\ni x\mapsto E_1(B_{1,x})\in \R$ is $C^\infty$. Moreover, if $\psi_{\eps,r}$ denotes the positive eigenfunction relative to $E_1(B_{\eps,r})$ 
and normalised to $1$ in $L^2(\B_\eps)$, then for any $\eps>0$ the mapping $\R\ni r\mapsto \psi_{\eps,r}\in L^2(\B_\eps)$ is $C^\infty$. 

\smallskip

\item[(c)] There exists $\varphi\in L^\infty(0,\infty)$ such that 
\begin{equation} \label{ej-nd} 
E_1(B_{1,x}) = -nx +x^2 \varphi(x) \quad \forall\, x >0. 
\end{equation} 

\smallskip

\item[(d)] Let $E_2^N>0$ denote the second eigenvalue of the Neumann Laplacian on $\B_1$. Then 
$$
E_2(B_{1,x})  = E_2^N +o(1), \qquad x\to 0. 
$$

\smallskip

\item[(e)] Let $\psi_{\eps,r}$ be as in part $({\rm b})$. Then 
for any $r_0>0$ there exists $\eps_0>0$ and a constant $K>0$ such that 
\begin{equation} \label{eq-kp}
\int_{\B_\eps} \big|\partial_r \psi_{\eps,r}(y)\big|^2\, dy  \ \leq \ K\, \eps^2 \qquad  \forall\, \eps\in (0,\eps_0), \ \ \forall\, r \in (0,r_0).
\end{equation} 
\end{enumerate}

\end{lemma} 

\begin{proof} The property (a) is easily obtained by dilations. From the compactness of the embedding $H^1(\B_1) \hookrightarrow L^2(\partial\B_1)$
it  follows that for any $\eta>0$ there exists $C_\eta$ such that 
\begin{equation}
\int_{\partial\B_1} |f|^2\, d\sigma \ \leq \ \eta \int_{\B_1} |\nabla f|^2\, dy + C_\eta \int_{\B_1} |f|^2\, dy
\end{equation}
holds true for all $f\in H^1(\B_1)$. Hence the mapping $x\mapsto B_{1,x}$
is a type (B) analytic family, which implies  (b) and (d). Moreover, since $B_{1,0}$ is the Neumann Laplacian on $\B_1$ whose first eigenvalue is simple and the associated eigenfunction is constant,
the analytic perturbation
theory gives
\begin{equation} \label{pt-1}
E_1(B_{1,x}) = -nx + \mathcal{O}(x^2), \qquad x\to 0. 
\end{equation} 
where we have used the fact that $|\partial\B_1| = n |\B_1|$. On the other hand, $E_1(B_{1,x}) = -x^2 +o(x^2)$ as $x\to +\infty$, see e.g. \cite{luzhu}. This together 
with \eqref{pt-1} gives (c).  

\smallskip

\noindent It remains to prove (e), which is done by rather direct computations. For the proof in the case $n=1$ we refer to
\cite[Lemma~4.7]{kp}. Consider now the case $n\geq 2$ and let 
$$
\lambda = \lambda(\eps,r):= \sqrt{-E_1(B_{\eps,r})}\ .
$$
By symmetry $\psi_{\eps,r}(y) = \phi_{\eps,r}(|y|)$, where $ \phi_{\eps,r}$ is a positive solution of 
\begin{equation}  \label{ev-eq}
-\partial_t^2 \phi_{\eps,r}(t) -\frac{n-1}{t}\,  \partial_t\phi_{\eps,r}(t) = -\lambda^2(\eps,r)\,  \phi_{\eps,r}(t),
\end{equation}
satisfying $\partial_t   \phi_{\eps,r}(t) |_{t=\eps} = r\,  \phi_{\eps,r}(\eps)$. Writing 
\begin{equation}
\phi_{\eps,r}(t) =   (\lambda\, t)^{-\nu}\,  v(\lambda\, t)  , \qquad \nu := \frac n2-1
\end{equation}
and $ s= \lambda t$, we find out that 
\begin{equation}
s^2 v''(s) +s\, v'(s) -(s^2+\nu^2)\, v(s) =0.
\end{equation}
By \cite[Sec.~9.6.1]{as} the solutions of the last equation are given by the modified Bessel functions
$I_\nu$ and $K_\nu$. Since $s^{-\nu} K_\nu(s) \notin H^1(\B_\eps)$, see \cite[Eqs.~(9.6.8-9) \&~(9.6.28)]{as}, it follows that
\begin{equation} \label{ef-eq1}
\phi_{\eps,r}(t) = \beta(\lambda,\eps)\, u_{\eps,r}(\lambda\, t) := \beta(\lambda,\eps)\,  (\lambda\, t)^{-\nu}\, I_\nu(\lambda\, t),
\end{equation}
where $\beta(\lambda,\eps)$ is chosen so that $\|\psi_{\eps,r}\|_{L^2(\B_\eps)}=1$. The latter condition implies 
\begin{equation} \label{eq-beta}
\omega_n\, \beta^2(\lambda,\eps) \int_0^{\eps\lambda} I^2_\nu(s)\, s\, ds =  \lambda^n\, .
\end{equation} 
To prove estimate \eqref{eq-kp} we use the identity
\begin{align} \label{eq-aux1}
\partial_r \psi_{\eps,r}(y) & = \partial_\lambda \phi_{\eps,r}(|y|) \, \frac{\partial\lambda(\eps,r)}{\partial r} \nonumber \\
& = -\frac{1}{2 \lambda(\eps,r)}\, \frac{\partial E_1(B_{\eps,r})}{\partial r}  \Big(\partial_\lambda \beta(\lambda,\eps) \, u_{\eps,r}(\lambda\, |y|)+ 
|y|\, \beta(\lambda,\eps)\,  u'_{\eps,r}(\lambda\, |y|) \Big), 
\end{align} 
where $ u'_{\eps,r}(s) = \partial_s  u_{\eps,r}(s)$. Differentiating equation \eqref{eq-beta} with respect to $\lambda$ gives 
\begin{align} \label{eq-beta'}
n \lambda^{n-1} - \omega_n\, \beta^2(\lambda,\eps)\, \eps^2 \lambda\, I^2_\nu(\eps\lambda) & = 2\, \omega_n\, \beta(\lambda,\eps)\, \partial_\lambda \beta(\lambda,\eps) 
\int_0^{\eps\lambda} I^2_\nu(s)\, s\, ds .
\end{align}
By \cite[Eq.~(9.6.10)]{as} we have
\begin{equation}  \label{eq-I}
I_\nu(s) = 2^{-\nu}\, s^\nu \, \sum_{k=0}^\infty \frac{4^{-k}\, s^{2k}}{k!\,  \Gamma(k+\nu+1)} \, .
\end{equation}
Keeping in mind that $2\nu = n-2$, it follows that 
\begin{equation} 
I^2_\nu(s) = a^2_\nu\, s^{n-2} +2a_\nu b_\nu\, s^n + o(s^n), \qquad s\to 0
\end{equation} 
holds true with
$$
 a_\nu=\frac{1}{2^\nu\, \Gamma(\nu+1)}, \qquad b_\nu= \frac{1}{2^{2+\nu}\, \Gamma(\nu+2)}\, .
$$
This together with \eqref{eq-beta} implies 
\begin{align*}
\omega_n\, \beta^2(\lambda,\eps)\, \eps^2 \lambda\, I^2_\nu(\eps\lambda) & =
\frac{\eps^2\, \lambda^{n+1}\, I^2_\nu(\eps\lambda)}{\displaystyle\int_0^{\eps\lambda} I^2_\nu(s)\, s\, ds}=
 \frac{\lambda^{n+1}\, \eps^2\, a_\nu^2\, (\eps\lambda)^{n-2} \big(1+c_\nu\, \eps^2\lambda^2 +o (\eps^2\lambda^2)\big)}{\frac 1n\  \eps^n\lambda^n\, a_\nu^2 \Big(1+\frac{n}{n+2}\  c_\nu\, \eps^2\lambda^2 +o (\eps^2\lambda^2)\Big)}\\
 & = n\, \lambda^{n-1}\Big(1+ \frac{2}{n+2}\ c_\nu \, \eps^2\lambda^2 +o(\eps^2\lambda^2) \Big),  \quad \eps\to 0,
\end{align*}
where 
$$
c_\nu = \frac{2 b_\nu}{a_\nu} = \frac{\Gamma(\nu+1)}{2\, \Gamma(\nu+2)} = \frac{1}{2(\nu+1)}\, .
$$
Here we have used the identity $\Gamma(z+1) =z\, \Gamma(z)$.  
Hence
$$
n \lambda^{n-1} - \omega_n\, \beta^2(\lambda,\eps)\, \eps^2 \lambda\, I^2_\nu(\eps\lambda)  = -\frac{n\, \lambda^{n+1}\, \eps^2}{(n+2)(\nu+1)} + 
o( \lambda^{n+1}\, \eps^2)
 \qquad \eps\to 0. 
$$
Using \eqref{eq-beta} and  \eqref{eq-I} again one easily verifies that 
\begin{align*}
\omega_n\, \beta (\lambda,\eps)\, \int_0^{\eps\lambda} I^2_\nu(s)\, s\, ds & = \frac{\sqrt{\omega_n}}{2^\nu \sqrt{n}}\ \eps^{\frac n2} \lambda^n + 
o\big(\eps^{\frac n2} \lambda^n\big),   \qquad \eps\to 0. 
\end{align*}
Equation \eqref{eq-beta'} thus gives 
\begin{equation} \label{beta'-uppb}
 \partial_\lambda \beta(\lambda,\eps)  = -\frac{n^{3/2}\, 2^{\nu-1}\, \lambda\, \eps^{2-\frac n2}}{\sqrt{\omega_n}\, (n+2)(\nu+1)}  \, +o( \lambda\, \eps^{2-\frac n2}),  \qquad \eps\to 0.
\end{equation}
Since 
$$
u_{\eps,r}'(s) = 2^{-\nu}\,  \, \sum_{k=1}^\infty \frac{4^{-k}\, 2k\, s^{2k-1}}{k!\,  \Gamma(k+\nu+1)}\ ,
$$
see \eqref{ef-eq1} and \eqref{eq-I}, the above estimates and a simple calculation show that the upper bound
\begin{equation} \label{eq-aux2}
\int_{\B_\eps} \Big(\partial_\lambda \beta(\lambda,\eps) \, u_{\eps,r}(\lambda\, |y|)+ 
|y|\, \beta(\lambda,\eps)\,  u'_{\eps,r}(\lambda\, |y|) \Big)^2\, dy \ \leq \ C_1\,  \lambda^2\, \eps^4
\end{equation} 
holds for all $\eps\in (0,\eps_0)$ and $\forall\, r \in (0,r_0)$ and with a constant $C_1$ depending only on $n, r_0$ and $\eps_0$. 
On the other hand, part (a) of the Lemma in combination with \eqref{pt-1} implies 
$$
\Big | \, \frac{\partial E_1(B_{\eps,r})}{\partial r} \, \Big|  \ \leq \ C_2 \, \eps^{-1} \qquad  \forall\, \eps\in (0,\eps_0), \ \ \forall\, r \in (0,r_0). 
$$
In view of \eqref{eq-aux1} and \eqref{eq-aux2} this proves estimate \eqref{eq-kp} for $n\geq 2$. 
\end{proof}

\section{\bf Model peak operator}\label{sec-teps}

\subsection{Problem setting} Throughout this section, we keep fixed a value of $a>0$.
Our goal is to study the properties of the operator $\Tilde S_1^{\eps,a}$ for $\eps\to 0$. In order to simplify the notation we denote 
\begin{equation}
  \label{eq-tea}
T_{\eps,a} := \Tilde S_1^{\eps,a}\, .
\end{equation}
Then $T_{\eps,a} $ is the self-adjoint operator in $L^2(V_{\eps,a})$ generated by the quadratic form
\[
t_{\eps,a}(u,u)=\Tilde s\,_1^{\eps,a}(u,u)= \int_{V_{\eps,a}} |\nabla u|^2dx- \int_{\partial_0 V_{\eps,a}} u^2\, d\sigma, \quad \cD(t_{\eps,a})=\Tilde H^1_0(V_{\eps,a}),
\]
see section \ref{sec-prelim} for the notation. 
We start with a technical result. Denote
\[
\cD_0(t_{\eps,a})=\big\{ u\in C^\infty(\overline{V_{\eps,a}}): \text{ $\exists\, b,c\in (0,a)$ such that $u(x)=0$ for $x_1<b$ and for $x_1>c$}\big\}.
\]

\begin{lemma}\label{lem-dens}
The subspace $\cD_0(t_{\eps,a})$ is dense in $\Tilde H^1_0(V_{\eps,a})$ in the norm of $H^1(V_{\eps,a})$.
\end{lemma}

\begin{proof} We provide a quite standard proof for sake of completeness.
First remark that the subspace $\cD^\infty:=\Tilde H^1_0(V_{\eps,a})\cap L^\infty(V_{\eps,a})$ is dense 
in $\Tilde H^1_0(V_{\eps,a})$ in the norm of $H^1(V_{\eps,a})$. Indeed, for $u\in \Tilde H^1_0(V_{\eps,a})$ and $k>0$
set $u_k:=\min\big\{\max\{u,-k\},k\big\}$, then $u_k\in L^\infty(V_{\eps,a})$ and it is standard to check that
$u_k\in \Tilde H^1_0(V_{\eps,a})$ and that $u_k$ converges to $u$ in $H^1(V_{\eps,a})$ as $k\to+\infty$.
Therefore, it is sufficient to check that any function from $\cD^\infty$ is the limit in $H^1(V_{\eps,a})$
of functions from $\cD_0(t_{\eps,a})$.

Let $u\in \cD^\infty$. Let $\chi:\R\to \R$ be an increasing $C^\infty$-function with $\chi(s)=0$ for $s<\frac{1}{2}$ and $\chi(s)=1$ for $s>1$.
For $\delta>0$ we denote by $u_\delta$ the function on $V_{\eps,a}$ given by
\[
u_\delta(x)=u(x)\,  \chi\Big( \frac{x_1}{\delta}\Big)\, \chi\Big( \frac{a-x_1}{\delta}\Big)\, .
\]
Then for some $C>0$ and small $\delta>0$ one has
\begin{multline*}
\|u-u_\delta\|^2_{H^1(V_{\eps,a})}  \leq C \int_{V_{\eps,a}} \big( |u|^2 +|\nabla u|^2)\, \Big(1-\chi\Big( \frac{x_1}{\delta}\Big)\, \chi\Big( \frac{a-x_1}{\delta}\Big)\Big)^2\, dx\\
+ \frac{C}{\delta^{2}} \int_{V_{\eps,a}: \, x_1< \delta} |u|^2 \, dx
+ \frac{C}{\delta^{2}} \int_{V_{\eps,a}: \, x_1>a-\delta} |u|^2 \, dx=:I_1+I_2+I_3.
\end{multline*}
By the monotone convergence theorem the term $I_1$ tends to zero as $\delta\to 0$. To estimate $I_2$ we remark that
\[
\int_{V_{\eps,a}: \, x_1< \delta} |u|^2 \, dx \le \|u\|^2_\infty \int_{V_{\eps,a}: \, x_1< \delta} 1 \, dx
=\|u\|^2_\infty \int_0^\delta \int_{\B_{\eps x_1^p}} dx'\, dx_1=O(\delta^{np+1}),
\]
and due to $np+1>2$ we see that $I_2$ tends to zero as $\delta\to 0$.
To estimate $I_3$ we first remark that almost everywhere one has
\[
u(x_1,x')=\int_a^{x_1} \partial_{x_1} u(t,x')dt,
\]
hence, using H\"older's inequality,
\[
\big|u(x_1,x')\big|^2\le (a-x_1)\int_{x_1}^a \big(\partial_{x_1} u(t,x')dx_1\big)^2 dt \le (a-x_1) \int_{x_1}^a \big|\nabla u(t,x')\big|^2 dt,
\]
and then
\begin{align*}
I_3&\le \dfrac{C}{\delta^2} \int_{V_{\eps,a}: \, x_1>a-\delta} (a-x_1) \int_{x_1}^a \big|\nabla u(t,x')\big|^2 dt\,dx\\
&= \dfrac{C}{\delta^2} \int_{a-\delta}^a (a-x_1) \int_{x_1}^a \int_{-\eps x_1^p}^{\eps x_1^p}  \big|\nabla u(t,x')\big|^2 dx'\, dt \, dx_1\\
& \le \dfrac{C}{\delta^2} \int_{a-\delta}^a (a-x_1) \int_{x_1}^a \int_{-\eps t^p}^{\eps t^p}  \big|\nabla u(t,x')\big|^2 dx'\, dt \, dx_1\\
&\le \dfrac{C}{\delta^2} \int_{a-\delta}^a (a-x_1)dx_1 \int_{a-\delta}^a \int_{-\eps t^p}^{\eps t^p}  \big|\nabla u(t,x')\big|^2 dx'\, dt\\
&=\dfrac{C}{2 } \int_{V_{\eps,a}: \, x_1>a-\delta} |\nabla u|^2dx,
\end{align*}
and the last term tends to $0$ for $\delta\to 0$ due to $|\nabla u|^2\in L^1(V_{\eps,a})$. Therefore, $u_\delta$ converges to $u$ in $H^1(V_{\eps,a})$.

As $u_\delta=0$ for $x_1<\delta/2$ and $x_1>a-\delta/2$, the preceding constructions show that the set 
\[
\cD_1(t_{\eps,a})=\big\{ u\in H^1(V_{\eps,a}): \text{ $\exists\,b,c\in (0,a)$ such that $u(x)=0$ for $x_1<b$ and for $x_1>c$}\big\}
\]
is dense in $\Tilde H^1_0(V_{\eps,a})$ in the norm of $H^1(V_{\eps,a})$. On the other hand, in the same norm $\cD_0(t_{\eps,a})$ is dense in $\cD_1(t_{\eps,a})$ using
the standard mollifying procedure.
\end{proof}

\noindent In view of Lemma~\ref{lem-dens} it follows by the min-max principle that for any $j\in\N$ one has
\begin{equation} \label{ej-vp0}
E_j(T_{\eps,a}) = \inf_{\substack{S\subset \cD_0(t_{\eps,a}) \\ {\rm dim\, } S =j} } \, \sup_{\substack{u\in S\\ u\neq 0}}\,  \frac{t_{\eps,a}(u,u)}{\quad \|u\|_{L^2(V_{\eps,a})}^2}\,.
\end{equation}

\subsection{Change of variables} \label{sec-subs0}
Let $(s,t)=(s,t_1,t_2,\dots,t_n) \in \Pi_{\eps}$, where 
\begin{equation} 
\Pi_{\eps} = (0,a)\times \B_\eps\, , \qquad \B_\eps\subset \R^{n}\, . 
\end{equation} 
Then $V_{\eps,a} = X (\Pi_\eps)$ for $X(s,t) = (s,t s^p)$, and the transform 
\begin{equation} 
u \mapsto \U\, u, \quad \U\, u(s,t) := u\big(X(s,t)\big)
\end{equation} 
maps $L^2(V_{\eps,a})$ unitarily on $L^2(\Pi_\eps, s^{np}\, ds\,dt)$.
We are going to study the quadratic form $q_\eps$ in $L^2(\Pi_\eps, s^{np}\, ds\,dt)$
given by 
$$
q_\eps(u,u) := t_{\eps,a} (\U^{-1} u, \U^{-1} u)
$$ 
with the domain $\cD(q_\eps)=\U \cD(t_{\eps,a})$. For this purpose, denote
\begin{multline}
\cD_0(q_\varepsilon):=\U\, \cD_0(t_{\eps,a})\\
\equiv  \big\{u\in C^\infty(\overline{\Pi_\eps}):\, \exists\, b,c\in (0,a) \text{ such that } u(s,t)=0 \text{ for } s<b \text{ and  for }s>c\big\},
\end{multline}
which is a core of $q_\eps$ by construction. Hence in view of \eqref{ej-vp0} one has
\begin{equation} \label{ej-vp2}
E_j(T_{\eps,a}) = \inf_{\substack{S\subset \cD_0(q_\eps) \\ {\rm dim\, } S =j} } \, \sup_{\substack{u\in S\\ u\neq 0}}\,  \frac{q_\eps(u,u)}{\quad \|u\|_{L^2(\Pi_\eps, s^{np}\, ds\,dt)}^2} \, ,
\end{equation} 
and a standard calculation then shows that for $u\in \cD_0(q_\varepsilon)$ there holds
\begin{align} \label{qeps}
q_\eps(u,u) & = \int_0^a\!\!\! \int_{\B_\eps}  \big\langle \nabla u,{\rm G}\,  \nabla u \big\rangle_{\R^N}\,  s^{pn}  dt ds 
 - \int_0^a\!\! \sqrt{1+p^2 \varepsilon^2 s^{2p-2}}  \int_{\partial\B_\eps} u^2   s^{p(n-1)} d\tau ds, 
\end{align}
where $d\tau$ denotes  the $(n-1)$-dimensional Hausdorff measure, and ${\rm G}$ is an $N\times N$ matrix given by
\[
{\rm G}= \left(\, 
 \begin{matrix}
1+p^2 |t|^2 s^{2p-2} &  p s^{2p-1}\,  t  \\
& \\
p s^{2p-1}\,t^T  &  s^{2p}\, \id
\end{matrix}\, \right)^{-1} .
\]
Here $\id$ stands for the $n\times n$ identity matrix. One checks directly that
\begin{equation} \label{eq-m}
{\rm G}= \left(\, 
 \begin{matrix}
1&  -\frac ps\,  t  \\
& \\
-\frac ps\,t^T  &  C
\end{matrix}\, \right)\, 
\quad \text{with} 
\quad 
C_{jk} = \begin{cases}
s^{-2p}+p^2\, t_j^2\, s^{-2} &  \text{ if} \quad j=k,\\
p^2 s^{-2}\, t_j t_k &   \text{ if} \quad  j\neq k.
\end{cases}
\end{equation}
Using the Young inequality and equation \eqref{eq-m} we find that for $\eps$ small enough
\begin{equation} \label{m-est}
\begin{aligned} 
& \big(1-n p \eps\big)\, |\partial_s u|^2 + \Big (\frac{1}{s^{2p}} -\frac{n \eps^2 p^2+\eps p}{s^2}\Big)\, |\nabla_t u|^2  \leq 
 \\
& \qquad \big\langle \nabla u,{\rm G}\,  \nabla u \big\rangle_{\R^N}  \leq  \big(1+n p \eps\big)\, |\partial_s u|^2 + \Big(\frac{1}{s^{2p}} +\frac{n \eps^2 p^2+\eps p}{s^2}\Big)\, |\nabla_t u|^2 \, .
\end{aligned}
\end{equation}

\smallskip

\noindent In what follows we will also need the transform
\begin{equation} \label{map-v}
u \mapsto (\V\, u)(s,t) = s^{-\frac{pn}{2}}\,  u(s,t),
\end{equation}
which maps $L^2(\Pi_\eps)$ unitarily onto $L^2(\Pi_\eps, s^{np} ds dt)$.

\subsection{Upper bound} \label{sec-upperb}
We start with a comparison between $T_{\eps,a}$ and the one-dimensional operator $L_{\eps,a}$.

\begin{lemma}\label{lem6}
There exist $c>0$, $c'>0$, $\eps_0>0$  such that
\begin{equation}
E_j(T_{\eps,a})\le (1+c\eps)E_j(L_{(1+c\eps)\eps,a}) + c'  \qquad  \forall \, j\in \N, \ \forall\, \eps\in(0,\eps_0).
\end{equation}
\end{lemma}

\begin{proof}
By equations \eqref{qeps} and \eqref{m-est}, for $u\in \cD_0(q_\eps)$ one has
\begin{align*}
q_\eps(u,u) \, \leq \, q_\eps^+(u,u) & :=  \int_0^a \int_{\B_\eps} \left( (1+np \eps) |\partial_s u|^2 +\left(\frac{1}{s^{2p}} + \frac{ p\eps+np^2 \eps^2}{s^2}\right)\, |\nabla_t u|^2 \right)\, s^{np} \, dt ds  \\
& \quad - \int_0^a  \int_{\partial\B_\eps}\, s^{p(n-1)}\, u^2\,  d\tau ds\, .
\end{align*}
A simple calculation then shows that
for $u\in \cD_0(r_\eps^+):=\V^{-1} \cD_0(q_\eps)\equiv \cD_0(q_\eps)$
there holds
\begin{align*}
r_\eps^+(u,u) & := q_\eps^+( \V\, u , \V\, u)  \\
&\ =   \int_0^a \int_{\B_\eps} \left( (1+np \eps) \Big(\partial_s u-\frac{n p u}{2s}\Big)^2 +\left(\frac{1}{s^{2p}} + \frac{p \eps+np^2 \eps^2}{s^2}\right)\, |\nabla_t u|^2 \right) \, dt ds \\
& \quad -\int_0^a  \frac{1}{s^p}\,   \int_{\partial\B_\eps} u^2\,  d\tau ds \, .
\end{align*}
Eq. \eqref{ej-vp2} implies then
\begin{equation}  \label{1-upperb0}
E_j(T_{\eps,a}) \le \inf_{\substack{S\subset \cD_0(r^+_\eps) \\ {\rm dim\, } S =j} } \, \sup_{\substack{u\in S\\ u\neq 0}}\,  \frac{r^+_\eps(u,u)}{\quad \|u\|_{L^2(\Pi_\eps)}^2} \, 
\end{equation} 
The integration by parts gives that for $u\in \cD_0(r^+_\eps)$ one has
\begin{equation} \label{per-partes}
\int_0^a u\,  \partial_s u\  \frac{ds}{s} = \int_0^a \frac{u^2}{2s^2}\, ds, 
\end{equation}
which implies that 
\begin{align*}
r_\eps^+(u,u) & =  \int_0^a \int_{\B_\eps} \left( (1+n p \eps) \Big(|\partial_s u|^2 +\frac{n^2p^2-2np}{4 s^2}\, u^2 \Big)+ \left(\frac{1}{s^{2p}} + \frac{p \eps+np^2 \eps^2}{s^2}\right)\, |\nabla_t u|^2 \right) dt ds \\
& \quad - \int_0^a  \frac{1}{s^p}\,   \int_{\partial\B_\eps} u^2\,  d\tau ds .
\end{align*}
Having in mind that due to \eqref{hardy}
\[
\int_0^a \Big(|\partial_s u|^2 +\frac{n^2 p^2-2n p}{4 s^2}\, u^2 \Big) ds\ge \int_0^a \Big(|\partial_s u|^2 -\frac{u^2}{4 s^2}\,  \Big) ds \ge  0,
\]
for $\eps$ small enough one can estimate, with a suitable $c>0$,
\begin{align*}
r_\eps^+(u,u) & \leq  (1+c\eps)  \int_0^a \int_{\B_\eps} \left( |\partial_s u|^2 -\frac{n^2p^2-2np}{4 s^2}\, u^2 +\frac{1}{s^{2p}}\, |\nabla_t u|^2 \right) dt ds  - \int_0^a \int_{\partial\B_\eps} u^2 \frac{d\tau ds}{s^p} \\
& = (1+c\eps)  \int_0^a \int_{\B_\eps} \left( |\partial_s u|^2 -\frac{n^2 p^2-2np}{4 s^2}\, u^2 \right)\, dt ds \\
& \quad + (1+c\eps) \int_0^a \frac{1}{s^{2p}} \left\{  \int_{\B_\eps} |\nabla_t u|^2 \, dt - \frac{s^p}{1+c\eps}\, \int_{\partial\B_\eps} u^2\,  d\tau \right\} ds\, .
\end{align*}
Note that the functional in the curly brackets is the quadratic form $b_{\eps,\rho(s,\eps)}$ as defined in section \ref{sec-1d}
with 
\[
\rho(s,\eps) =(1+c\eps)^{-1} s^p. 
\]
Denote by $\psi\equiv \psi_{\eps, \rho(s,\eps)}$  the positive normalized eigenfunction of $B_{\eps,\rho(s,\eps)}$ relative to the eigenvalue $E_1(B_{\eps,\rho(s,\eps)})$. 

\smallskip

\noindent Now let $S \subset C_0^\infty(0,a)$ be a linear subspace with dimension $j$ and define
\begin{equation} \label{eq-S}
\Tilde S = \big\{ u:\Pi_\eps\to \R:\, u(s,t)= f(s)\, \psi_{\eps,\rho(s,\eps)}(t), \, f\in S\big\} .
\end{equation} 
Then $ \dim \Tilde S=j$ and $\Tilde S\subset \cD_0(r_\eps^+)$ due to Lemma~\ref{lem-1}. Hence for $u\in \Tilde S$ one has 
\begin{gather*}
\|u\|_{L^2(\Pi_\eps)} = \|f\|_{L^2(0,a)},\quad
\int_{\B_\eps} |\nabla_t u|^2 \, dt - \frac{s^p}{1+c\eps}\, \int_{\partial\B_\eps} u^2\,  d\tau  = E_1(B_{\eps,\rho(s,\eps)}) f(s)^2\, .
\end{gather*}
Moreover,
\begin{align*} 
 \int_0^a\!\! \int_{\B_\eps}\!\! \Big( |\partial_s u|^2 +\frac{n^2 p^2-2np}{4 s^2}\, u^2 \Big) dt ds = \int_0^a\! \Big[  |f'|^2 + \Big(\frac{n^2p^2-2np}{4s^2} + \int_{\B_\eps} \!\! |\partial_s \psi_{\eps, \rho(s,\eps)}|^2 dt\Big) f^2 \Big] ds \, .
\end{align*}
Using \eqref{eq-kp} we see that there exists $K>0$ such that for $\eps$ small enough we have
\begin{multline} \label{eq-kp2} 
\int_{\B_\eps} \big|\partial_s \psi_{\eps,\rho(s,\eps)}(t)\big|^2\, dt
=\int_{\B_\eps} \Big( \partial_\rho \psi_{\eps,\rho}(t) \big|_{\rho= \rho(\eps,s)} \, \dfrac{\partial \rho(\eps,s)}{\partial s}\Big)^2 \, dt\\
=\dfrac{p^2 s^{2p-2}}{(1+c\eps)^2}\int_{\B_\eps} \Big( \partial_\rho \psi_{\eps,\rho}(t) \big|_{\rho= \rho(\eps,s)}\Big)^2 \, dt
\ \leq \ K\, \frac{p^2  s^{2p-2}}{(1+c\eps)^2}\ \eps^2\,  <\eps \quad \forall\,  s\in(0,a).
\end{multline}
Hence
\[
r_\eps^+(u,u)  \leq (1+c\eps)  \int_0^a \left[ |f'|^2 +\left(\frac{n^2p^2-2np}{4s^2} +\eps + \frac{E_1(B_{\eps,\rho(s,\eps)})}{s^{2 p}} \right) f^2  \right] ds\, .
\]
To continue we apply Lemma \ref{lem-1} which implies that there exists $c_0>0$, independent of $\eps$,  such that 
\begin{align*}
\frac{E_1(B_{\eps,\rho(s,\eps)})}{s^{2p}} & = \frac{E_1(B_{1,\eps \rho(s,\eps)})}{\eps^2 s^{2p}}\leq \frac{-n\, \eps\, \rho(s,\eps) +c_0\, \eps^2 \rho^2(s,\eps) }{\eps^2 s^{2p}} \\
& \leq\  -\frac{n}{(1+c\eps)\, \eps\, s^p} +\frac{c_0}{(1+c\eps)^2}
\end{align*}
This implies that the inequality
\begin{align*} \label{b+upperb}
\frac{r^+_\eps(u,u)}{ \quad \|u\|_{L^2(\Pi_\eps)}^2}  &\ \leq\   \frac{(1+c\eps) \displaystyle\int_0^a \left[ |f'|^2 +\left(\frac{n^2p^2-2np}{4s^2} -\frac{n}{(1+c\eps)\, \eps\, s^p} \right) f^2  \right] ds}{\|f\|_2^2} \\
&\qquad  + \frac{c_0}{1+c\eps} +\eps(1+c\eps)\\
&=(1+c\eps) \dfrac{\big\langle f,L_{(1+c\eps)\, \eps,a}f\big\rangle_{L^2(0,a)}}{\|f\|_2^2}+ \frac{c_0}{1+c\eps} +\eps(1+c\eps).
\end{align*}
holds for each $u\in\Tilde S$. Therefore,
\begin{multline*}
\inf_{\substack{S\subset \cD_0(r^+_\eps) \\ {\rm dim\, } S =j} } \, \sup_{\substack{u\in S\\ u\neq 0}}\,  \frac{r^+_\eps(u,u)}{\quad \|u\|_{L^2(\Pi_\eps)}^2}
\le
\inf_{\substack{S\subset C_0^\infty(0,a) \\ {\rm dim\, } S =j} } \, \sup_{\substack{u\in \Tilde S\\ u\neq 0}}\,  \frac{r^+_\eps(u,u)}{\quad \|u\|_{L^2(\Pi_\eps)}^2}\\
\le (1+c\eps) \inf_{\substack{S\subset C_0^\infty(0,a) \\ {\rm dim\, } S =j} } \, \sup_{\substack{f\in  S\\ f\neq 0}}\,\dfrac{\big\langle f,L_{(1+c\eps)\, \eps,a}f\big\rangle_{L^2(0,a)}}{\|f\|_2^2}
+ \frac{c_0}{1+c\eps} +\eps(1+c\eps)\\
=(1+c\eps)E_j(L_{(1+c\eps)\, \eps,a})+ \frac{c_0}{1+c\eps} +\eps(1+c\eps),
\end{multline*}
and the substitution into \eqref{1-upperb0} concludes the proof.
\end{proof}

\noindent A combination of Lemma~\ref{lem6} with Lemma~\ref{lem7} gives then the main result of this subsection:
\begin{prop}\label{prop7}
There exist $K>0$, $k>0$ and $\eps_0>0$ such that
\[
E_j(T_{\eps,a})\le (1-k\eps)\, \eps^{\frac{2}{2-p}}E_j(A_1) \quad
\forall\, \eps\in(0,\eps),\  1\leq j \leq \cN\big(A_1,-K\eps^{\frac{p}{p-2}}\big).
\]
\end{prop}

\smallskip

\subsection{Lower bound} 

\begin{lemma}\label{lem9}
There exist $\eps_0>0$, $b>0$ and $B>0$ such that for all $j\in\big\{ 1,\dots,\cN (T_{\eps,a},-B\big)\}$
one has
\begin{equation}
   \label{eq-low3}
E_j(T_{\eps,a})\ge (1-b\eps)E_j(L_{(1-b\eps)^2\eps,a})-B.
\end{equation}
\end{lemma}

\begin{proof}
There exist $c>0$ and $\eps_0>0$ such that for all $\eps\in(0,\eps_0)$ and $s\in(0,a)$ there holds
\begin{align*}
\frac{1}{s^{2p}} - \frac{p\eps+np^2 \eps^2}{s^2} \,   &  \ge
\frac{1-(p\eps+np^2 \eps^2)\, a^{2p-2}}{s^{2p}}
\,  \geq \, \frac{1-c\eps }{s^{2p}}, \\
\sqrt{1+p^2\, \eps^2 s^{2p-2}}  & \le   \sqrt{1+p^2\, \eps^2 a^{2p-2}}\le \frac{1}{1-c\eps}\, .
\end{align*}
Combining the two inequalities with \eqref{m-est} we estimate the quadratic form $q_\eps$ from below as follows:
\begin{align}
q_\eps(u,u)  \, \geq\, q^-_\eps(u,u) & := (1-c\eps) \int_0^a \int_{\B_\eps} \big( |\partial_s u|^2+ s^{-2p}\, |\nabla_t u|^2 \big)\, s^{np} \, dt ds \nonumber\\
& \quad -\frac{1}{1-c\eps} \int_0^a \int_{\partial\B_\eps} u^2\,  s^{p(n-1)} d\tau  ds\, 
\quad \forall\, u\in\cD_0(q_\eps).
    \label{a-lowerb}
\end{align}
By construction we then have 
\[
E_j(T_{\eps,a})\ge  \inf_{\substack{S\subset \cD_0(q_\eps) \\ {\rm dim\, } S =j} } \, \sup_{\substack{u\in S\\ u\neq 0}}\,  \frac{q^-_\eps(u,u)}{\quad \|u\|_{L^2(\Pi_\eps,s^{np}ds\,dt)}^2}.
\]
Now consider the quadratic form $r^-_\eps$ in $L^2(\Pi_\eps)$ given by $r^-_\eps(u,u)=q^-_\eps(\V u,\V u)$ defined on
$\cD_0(r^-_\eps)=\V^{-1} \cD_0(q_\eps)\equiv \cD_0(q_\eps)$. Hence
\begin{equation}
  \label{low-eq00}
E_j(T_{\eps,a})\ge  \inf_{\substack{S\subset \cD_0(r^-_\eps) \\ {\rm dim\, } S =j} } \, \sup_{\substack{u\in S\\ u\neq 0}}\,  \frac{r^-_\eps(u,u)}{\quad \|u\|_{L^2(\Pi_\eps)}^2}.
\end{equation}
The direct substitution in combination with \eqref{per-partes} shows that
\begin{align}
   \label{eq-b}
r^-_{\eps,\delta}(u,u) & = (1-c\eps) \int_0^a \int_{\B_\eps} \left( |\partial_s u|^2+\frac{n^2p^2-2np}{4 s^2} \ u^2+ s^{-2p}\, |\nabla_t u|^2 \right)\,  dt ds \\
& \quad -\frac{1}{1-c\eps} \int_0^a \frac{1}{s^{p}} \int_{\partial\B_\eps} u^2\,  d\tau ds \nonumber \\
& = (1-c\eps) \int_0^a \int_{\B_\eps} \left( |\partial_s u|^2+\frac{n^2p^2-2np}{4 s^2} \ u^2 \right)\,  dt ds   \nonumber \\
& \quad +  (1-c\eps) \int_0^a \frac{1}{s^{2p}} \left\{  \int_{\B_\eps} |\nabla_t u|^2 \, dt - \frac{s^p}{(1-c\eps)^2}\, \int_{\partial\B_\eps} u^2\,  d\tau \right\}ds\, .   \nonumber
\end{align}
The expression in the curly brackets is the quadratic form $b_{\eps,\rho(s,\eps)}$ with 
\begin{equation} \label{rs}
\varrho(s,\eps) = \frac{s^p}{(1-c\eps)^2} \in (0,M), \quad M:=\frac{a^p}{(1-c \eps_0)^2}, \quad  \eps\in(0,\eps_0),
\end{equation}
see section \ref{sec-1d}. Let $\psi_{\eps, \varrho(s,\eps)}$ be the positive normalized eigenfunction of $B_{\eps,\varrho(s,\eps)}$ relative to the eigenvalue $E_1(B_{\eps,\varrho(s,\eps)})$. 
We decompose each $u\in \cD_0(r^-_\eps)$ as
\begin{equation}
      \label{eq-v} 
u = v + w,  \ \text{ where } \ 
v(s,t)=  \psi_{\eps, \varrho(s,\eps)}(t) \, f(s), \quad  f(s) := \int_{\B_\eps} u(s,t)\, \psi_{\eps, \varrho(s,\eps)}(t)\, dt. 
\end{equation}
Notice that by construction we have $f\in C^\infty_0(0,a)$. Furthermore,
\begin{gather} \label{orth}
\int_{\B_\eps} w(s,t)\, \psi_{\eps, \varrho(s,\eps)}(t)\, dt = 0 \quad \forall\,  s\in (0,a),\\
\|f\|^2_{L^2(0,a)} + \|w\|^2_{L^2(\Pi_\eps)}=\|u\|^2_{L^2(\Pi_\eps)},  \label{eq-normuw}
\end{gather}
and the spectral theorem implies that 
\begin{equation}
   \label{eq-utw}
\int_{\B_\eps} |\nabla_t u|^2 \, dt - \varrho(s,\eps)  \int_{\partial\B_\eps} u^2\,  d\tau \, \geq \, E_1(B_{\eps,\varrho(s,\eps)})\, f(s)^2 +E_2(B_{\eps,\varrho(s,\eps)})\, \int_{\B_\eps} w^2\, dt \, .
\end{equation}
Recall, see Lemma \ref{lem-1}(c), that one can find a constant $c_1>0$ such that
\[
E_1(B_{1,x})= -n\, x+\cO(x^2)> -\dfrac{n\, x}{1-c_1 x} \qquad \text{for small $x>0$.}
\]
By Lemma \ref{lem-1}(a) we have  $E_1(B_{\eps,\varrho(s,\eps)})=\eps^{-2}E_1(B_{1,\eps \varrho (s,\eps)})$, and $\eps \varrho (s,\eps)\in[0,M\eps]$.
By adjusting the value of $\eps_0$ we conclude that there exists $c_2>0$ such that for all $\eps\in(0,\eps_0)$ and $s\in(0,a)$ it holds
\[
\dfrac{E_1(B_{\eps,\varrho(s,\eps)})}{s^{2p}}=\dfrac{E_1(B_{1,\eps \varrho(s,\eps)})}{\eps^2 s^{2p}}\ge -\dfrac{n\, \eps \varrho(s,\eps)}{\eps^2 s^{2p}\big(1-c_2\eps \varrho(s,\eps)\big)}
\ge -\dfrac{n}{\eps(1-c_2 \eps) s^{p}}.
\]
In a similar way, using the fact that $E_2(B_{1,x})= E_2^N+ O(x)\ge A_0>0$ for small $x$, see Lemma \ref{lem-1}(d), we conclude 
that if $\eps_0>0$ is sufficiently small, then
for all $s\in (0,a)$ and $\eps\in(0,\eps_0)$ there holds
\[
\dfrac{E_2(B_{\eps,\varrho(s,\eps)})}{s^{2p}} = \dfrac{E_2(B_{1,\eps \varrho(s,\eps)})}{\eps^2 s^{2p}}\, \geq\,  \frac{A_0}{\eps^2 s^{2p}}.
\]
Inserting these eigenvalue estimates into \eqref{eq-utw} we arrive the inequality
\[
\int_{\B_\eps} |\nabla_t u|^2 \, dt - \varrho(s,\eps) \int_{\partial\B_\eps} u^2\,  d\tau \ \, \geq \, -\dfrac{n\, f(s)^2}{\eps(1-c_2 \eps) s^{p}} +\frac{1}{\eps^2 s^{2p}}\, \int_{\B_\eps} w^2\, dt \, ,
\]
valid for all $u\in \cD_0(r^-_\eps)$. The substitution of the last inequality into \eqref{eq-b} shows that
one can find $k>0$ such that for all $\eps\in(0,\eps_0)$ and $u\in \cD_0(r^-_\eps)$ there holds
\begin{align}
r^-_{\eps}(u,u)  & \geq  (1-c\eps) \int_0^a \int_{\B_\eps} \left( |\partial_s u|^2+\frac{n^2p^2-2np}{4 s^2} \, u^2 \right) dt ds \label{lowerb-b}\\
&\quad  +  (1-c\eps)\int_0^a \int_{\B_\eps}  \frac{w^2}{\eps^2 s^{2p}}\, dt ds-\int_0^a \frac{n}{(1-k\eps) \eps s^p}\,f^2 \, ds.  \nonumber 
\end{align}
In the sequel, for the sake of brevity we will adopt the notation $\psi := \psi_{\eps, \varrho(s,\eps)}$ and 
$$
\psi_s:=\partial_s \psi, \qquad v_s:=\partial_s v, \qquad w_s:=\partial_s w\, .
$$
Let us study the first term on the right hand side of \eqref{lowerb-b}. Using \eqref{orth} we get
\begin{align}
     \label{us}
\int_0^a\!\! \int_{\B_\eps} \left( |\partial_s u|^2+\frac{n^2p^2-2np}{4 s^2} \, u^2 \right) dt ds  & = \int_0^a\!\! \int_{\B_\eps} \Big( v_s^2+\frac{n^2p^2-2np}{4 s^2} \, v^2 \Big) dt ds \nonumber \\
& \quad +\int_0^a\!\! \int_{\B_\eps} \left( w_s^2+\frac{n^2p^2-2np}{4 s^2} \, w^2 \right) dt ds \\
& \quad  + 2 \int_0^a\!\! \int_{\B_\eps} v_s\, w_s\, dt\, ds. \nonumber
\end{align}
Since $\psi$ is normalized, one has
\[
\int_{\B_\eps} \psi \psi_sdt=0,
\]
and the first term on the right-hand side of \eqref{us} can be bounded from below as follows; 
\begin{align}
\int_0^a\!\!\! \int_{\B_\eps}\!\!\! \Big( v_s^2+\frac{n^2p^2-2np}{4 s^2} \, v^2 \Big) dt ds & = \int_0^a \left[\,  |f'|^2 + \Big(\frac{n^2p^2-2np}{4s^2} + \int_{\B_\eps} \!\! \big| \psi_s \big|^2 dt\Big) f^2 \right] ds \nonumber \\
& \geq  \int_0^a \left[ \, |f'|^2 + \frac{n^2p^2-2np}{4s^2}\, f^2 \right] ds \label{vs-lowerb}
\end{align}
In order to estimate the last two terms in \eqref{us} we note that  
\begin{equation} \label{crossed}
 \int_0^a \int_{\B_\eps} v_s\, w_s\, ds dt  =  \int_0^a \int_{\B_\eps} f'\, \psi\, w_s \, dtds+  \int_0^a f \int_{\B_\eps} \psi_s\, w_s\, dt\, ds.
\end{equation}
Then, in view of \eqref{orth} 
\[
\int_{\B_\eps} \psi\, w_s \, dt=-\int_{\B_\eps}  \, \psi_s  w\, dt. 
\] 
Hence, using the Cauchy-Schwarz inequality,
\[
\Big | \int_0^a \int_{\B_\eps} f'\, \psi\, w_s \, dtds \Big |  = \Big | \int_0^a \int_{\B_\eps} f'\, \psi_s\,w\,   dt\,ds\, \Big | \leq \int_0^\delta |f'|^2  \int_{\B_\eps}   \psi_s^2\, dtds
 + \|w\|_{L^2(\Pi_\eps)}^2. 
\]
To estimate the last term in \eqref{crossed} we use again the Young inequality;
\[
\Big |   \int_0^a f \int_{\B_\eps} \psi_s\, w_s\, dt\, ds      \, \Big |  \ \leq\  \frac{1}{\eps} \int_0^a f^2  \int_{\B_\eps}  \psi_s^2\, dt\,ds + \eps \|w_s\|_{L^2(\Pi_\eps)}^2. 
\]
By  \eqref{eq-kp} and \eqref{rs} there is $K>0$ such that for $\eps\in(0,\eps_0)$ and $ s\in(0,a)$
 one has
\[
\int_{\B_\eps}  \psi_s^2\, dt \, \leq \, K\eps^2\le\eps.
\]
Putting the above estimates together we obtain the upper bound
\[
\Big |   \int_0^a \int_{\B_\eps} v_s\, w_s\, ds dt   \, \Big| \, \leq \,  \eps \int_0^a |f'|^2\, ds + K\eps \int_0^a f^2\, ds + \|w\|_{L^2(\Pi_\eps)}^2  + \eps\|w_s\|_{L^2(\Pi_\eps)}^2. 
\]
In view of \eqref{us} and \eqref{vs-lowerb} it follows that there exists $C>0$ such that for all $\eps\in(0,\eps_0)$ and $u\in\cD_0(r^-_\eps)$
one has
\begin{align*} 
\int_0^a \int_{\B_\eps} \left( |\partial_s u|^2+\frac{n^2p^2-2np}{4 s^2} \, u^2 \right) dt ds & \, \geq \,   \int_0^a \left[ \, (1-\eps)\, |f'|^2 + \Big(\frac{n^2p^2-2np}{4s^2} - C \Big) f^2 \right] ds \\
& +  \int_0^a \!\!\!\int_{\B_\eps}\!\!\left[ \, (1-\eps)\, w_s^2 + \Big(\frac{n^2p^2-2np}{4s^2} - C \Big) w^2 \right] dt ds\, .
\end{align*}
By \eqref{lowerb-b} this in turn gives
\begin{align*}
r^-_\eps(u,u)  & \ge   (1-c\eps)\int_0^a \left[ \, (1-\eps)\, |f'|^2 + \left(\frac{n^2p^2-2np}{4s^2} - C \right) f^2 \right] ds \\
& \quad +  (1-c\eps)\int_0^a \int_{\B_\eps}\left[ \, (1-\eps) w_s^2 + \left(\frac{n^2p^2-2np}{4s^2} - C \right) w^2 \right] dt ds\\
& \quad + (1-c\eps)\int_0^a \int_{\B_\eps}  \frac{A_0 w^2}{\eps^2 s^{2p}}\, \, dt ds-\int_0^a \frac{n\, f^2}{(1-k\eps)^2 \eps s^p}\,\, \, ds,
\end{align*}
and using the norm equality \eqref{eq-normuw} one may rewrite
\begin{equation}
   \label{eq-reps0}
\begin{aligned}
r^-_\eps(u,u) +(1-c\eps)C \|u\|^2_{\Pi_\eps}  & \ge   (1-c\eps)\int_0^a \left[ \, (1-\eps)\, |f'|^2 + \left(\frac{n^2p^2-2np}{4s^2} \right) f^2 \right] ds \\
 &\quad+  (1-c\eps)\int_0^a \int_{\B_\eps}\left[ \, (1-\eps) w_s^2 + \left(\frac{n^2p^2-2np}{4s^2} \right) w^2 \right] dt ds\\
 &\quad + (1-c\eps)\int_0^a \int_{\B_\eps}  \frac{A_0 w(s,t)^2}{\eps^2 s^{2p}}\, \, dt ds-\int_0^a \frac{n\, f^2(s)}{(1-k\eps)^2 \eps s^p}\,\, \, ds.
\end{aligned}
\end{equation}

\smallskip
\noindent Next we notice that due to the Hardy inequality \eqref{hardy} we have
\[
\int_0^a \dfrac{g(s)^2}{s^2}\, ds \le \frac{4}{(np-1)^2 }
\int_0^a \left[g'(s)^2 + \frac{n^2p^2-2np}{4s^2}\ g(s)^2 \right] ds \qquad \forall\,  g\in C^\infty_0(0,a).
\]
Therefore, there exists $c_0>0$ such that for all $\eps\in(0,\eps_0)$ and $u\in \cD_0(r^-_\eps)$ one has
\begin{align*}
\int_0^a\Big[ \, (1-\eps)\, |f'|^2 + \Big(\frac{n^2p^2-2np}{4s^2} \Big) f^2 \Big] ds&\ge (1-c_0 \eps) \int_0^a\Big[ |f'|^2 + \Big(\frac{n^2p^2-2np}{4s^2} \Big) f^2 \Big] ds,\\
\int_0^a\!\!\! \int_{\B_\eps} \!\! \Big[ \, (1-\eps) w_s^2 + \Big(\frac{n^2p^2-2np}{4s^2} \Big) w^2 \Big] dt ds & \ge
(1-c_0 \eps) \int_0^a\!\!\!  \int_{\B_\eps}\!\! \Big[ w_s^2 + \Big(\frac{n^2p^2-2np}{4s^2} \Big) w^2 \Big] dt ds.
\end{align*}
We thus conclude that there exist $b>0$ and $B>0$
such that for all $\eps\in(0,\eps_0)$ and all $u\in\cD_0(r^-_\eps)$ there holds
\begin{align*}
r^-_\eps(u,u) + B\|u\|^2_{L^2(\Pi_\eps)}  &\ge (1-b\eps)\int_0^a \left[ \, \, |f'|^2 + \left(\frac{n^2p^2-2np}{4s^2} - \frac{n}{(1-b\eps)^2 \eps s^p} \right) f^2 \right] ds\\
&\quad + (1-b\eps)\int_0^a \int_{\B_\eps}\left[ \, w_s^2 + \left(\frac{n^2p^2-2np}{4s^2} + \frac{A_0}{\eps^2 s^{2p}} \right) w^2 \right] dt ds.
\end{align*}
Assuming $\eps_0$ sufficiently small we have
\[
\frac{n^2p^2-2np}{4s^2} + \frac{1}{\eps^2 s^{2p}} \ge 0 \qquad \forall\,  s\in (0,a). 
\]
Therefore,
\begin{equation}
   \label{reps-b}
r^-_\eps(u,u)+ B\|u\|^2_{L^2(\Pi_\eps)}\ge (1-b\eps) \int_0^a \left[ \, \, |f'|^2 + \left(\frac{n^2p^2-2np}{4s^2} - \frac{n}{(1-b\eps)^2 \eps s^p} \right) f^2 \right] ds.
\end{equation}
Note that by construction of $f$ and $w$ and the norm equality \eqref{eq-normuw} the map $u\mapsto(f,w)$ extends to a unitary map $\Psi:L^2(\Pi_\eps)\to L^2(0,a)\times \cH$,
where $\cH$ is a closed subspace of $L^2(\Pi_\eps)$. Let $R^-_\eps$ be the self-adjoint operator in $L^2(\Pi_\eps)$ generated by the closure of $r^-_\eps$,
then the inequality \eqref{low-eq00} means that
\begin{equation}
  \label{treps}
E_j(T_{\eps,a})\ge E_j(R^-_\eps)  \qquad \forall\, j\in\N. 
\end{equation}
On the other hand, let $h_\eps$ be the quadratic form in $L^2(0,a)\times \cH$ defined as the closure of
the form
\[
C_0^\infty(0,a)\times \cH\ni (f,w) \mapsto \int_0^a \left[ \, \, |f'|^2 + \left(\frac{n^2p^2-2np}{4s^2} - \frac{n}{(1-b\eps)^2 \eps s^p} \right) f^2 \right]ds,
\]
then the corresponding self-adjoint operator in $L^2(0,a)\times \cH$ is $H_\eps=L_{(1-b\eps)^2 \eps, a}\oplus 0$.
The inequality \eqref{reps-b} reads as
\[
r^-_\eps(u,u)+ B\|u\|^2_{L^2(\Pi_\eps)}\ge (1-b\eps)\,  h_\eps(\Psi u, \Psi u), \quad u\in \cD_0(r^-_\eps)
\]
which by the min-max principle implies that $E_j(R^-_\eps)+B\ge (1-b\eps)E_j(H_\eps)$.
Assume now that $j\in\{1,\dots,\cN(T_{\eps,a},-B)\}$. Then $E_j(T_{\eps,a})<-B$ and $E_j(R^-_\eps)+B<0$, which shows that
for the same $j$ one has $E_j(H_\eps)<0$, and then $E_j(H_\eps)=E_j(L_{(1-b\eps)^2 \eps,a})$.
\end{proof}

\noindent Now we can state the main result of the subsection.

\begin{prop}\label{prop10}
Let $a>0$, then there exist $K>0$, $k>0$ and $\eps_0>0$ such that
\[
E_j(T_{\eps,a})\ge (1+k\eps)\, \eps^{\frac{2}{p-2}} \, E_j(A_1) - K
\qquad \forall\, \eps\in(0,\eps_0), \ 1\leq j \leq \cN(T_{\eps,a},-K).
\]
\end{prop}

\begin{proof}
Due to \eqref{eq-1dd} and \eqref{eq-hom0} one has,
for any $j\in\N$ and a suitably chosen $k>0$,
\[
E_j(L_{(1-b\eps)^2\eps,a})\ge E_j(A_{(1-b\eps)^2\eps})=(1-b\eps)^{\frac{4}{p-2}} \, \eps^{\frac{2}{p-2}} E_j(A_1)
\ge (1+k \eps)\, \eps^{\frac{2}{p-2}} E_j(A_1)\, .
\]
The substitution of the above lower bound into the result of Lemma~\ref{lem9} completes the proof.
\end{proof}

\section{\bf Proof of Proposition \ref{prop-main} }\label{sec-qeb}

\noindent
In order to simplify the notation, for $b>0$ and $\eps>0$, we denote 
$$
\ell:=b \, \eps^{\frac{1}{1-p}}. 
$$
Then, the operators $Q_{\eps,b}$ and $\Tilde Q_{\eps,b}$ defined in \eqref{Q-eps} are generated by the quadratic forms
\begin{align*}
q_{\eps,b}(u,u)&=\int_{V_{\eps,\ell}} |\nabla u|^2dx- \int_{\partial_0 V_{\eps,\ell}} u^2\, ds, \quad \cD(q_{\eps,b})=H^1(V_{\eps,\ell}),  \ \ \text{and} \\
\Tilde q_{\eps,b}(u,u)&=q_{\eps,b}(u,u), \quad \cD(\Tilde q_{\eps,b})=\Tilde H^1_0(V_{\eps,\ell})
\end{align*}
respectively, with $V_{\eps,\ell}$  defined in \eqref{def-veps}.  Note that the domain inclusions imply the obvious inequalities
\begin{align}
     \label{ineq1}
E_j(Q_{\eps,b}) & \le E_j(\Tilde Q_{\eps,b}) \quad \text{ for all } j\in\N,\\
     \label{ineq2}
E_j(\Tilde Q_{\eps,b}) & \le E_j(T_{\eps,a}) \quad \text{ for all $a\le \ell$ and $j\in\N$,}
\end{align}
where the operator $T_{\eps,a}$ is defined in \eqref{eq-tea}. Let us give a lower bound for $Q_{\eps,b}$:

\begin{lemma}\label{lem11}
Let $a>0$. Then there exist $B>0$, $k>0$ and $\eps_0>0$ such that
\[
E_j(Q_{\eps,b})\ge E_j(T_{\eps,a})-k \qquad \forall\,  \eps\in(0,\eps_0), \ \ 1\leq  j\leq \cN(Q_{\eps,b},-B/\eps).
\]
\end{lemma}

\begin{proof}
One may assume from the very beginning that $a\le \ell$. Let $\phi_1$ and $\phi_2$ be two smooth functions on $\R$ with the following properties:
\[
\phi_1^2+\phi_2^2=1, \quad \phi_1(s)=0 \text{ for $s>a$}, \quad \phi_2(s)=0 \text{ for $s<a/2$.}
\]
We set
\[
k:=\|\phi_1'\|_\infty^2+\|\phi_2'\|_\infty^2,
\quad
\chi_j(x):=\phi(x_1), \quad j\in\{1,2\}.
\]
By a direct computation, for any $u\in\cD(q_{\eps,b})$ there holds
\begin{align*}
q_{\eps,b}(u,u) & =q_{\eps,b}(\chi_1 u,\chi_1 u)+q_{\eps,b}(\chi_2 u,\chi_2 u)
-\int_{V_{\eps,\ell}} \big(|\nabla \chi_1|^2+|\nabla \chi_2|^2\big)\, u^2\, dx\\
& \ge q_{\eps,b}(\chi_1 u,\chi_1 u)+q_{\eps,b}(\chi_2 u,\chi_2 u) - k \|u\|^2_{L^2(V_{\eps,\ell})}.
\end{align*}
Denote
\begin{align*}
r_\eps(u,u)&=\int_{W_\eps} |\nabla u|^2dx- \int_{\partial_0 W_\eps} u^2\, ds, \quad \cD(r_\eps)=\Tilde H^1_0(W_\eps),\\
W_\eps&=\big\{(x_1,x'):\, x_1\in(a/2,\ell), \, |x'| < \eps\, x_1^p\big\}\subset\R^N,\\
\partial_0 W_\eps&=\big\{(x_1,x' ):\, x_1\in(a/2,\ell), \, |x'|= \eps\, x_1^p \big\}\subset \partial V_{\eps,\ell},\\
\Tilde H^1_0(W_\eps)&=\big\{u\in H^1(W_\eps):\, u(\ell,\cdot)= u(a/2,\cdot)=0\big\},
\end{align*}
and let $R_\eps$ be the self-adjoint operator acting in $L^2(W_\eps)$ and generated by $r_\eps$.
Then one has
\begin{align*}
\|\chi_1 u\|_{L^2(V_{\eps,\ell})}&=\|\chi_1 u\|_{L^2(V_{\eps,a})}, & \|\chi_2 u\|_{L^2(V_{\eps,\ell})}&=\|\chi_2 u\|_{L^2(W_\eps)},\\
q_{\eps,b}(\chi_1 u,\chi_1 u)& =t_{\eps,a}(\chi_1 u,\chi_1 u),&  q_{\eps,b}(\chi_2 u,\chi_2 u)&= r_\eps(\chi_2 u,\chi_2 u),
\end{align*}
and for any $j\in\N$ the min-max principle gives
\begin{equation}
\begin{aligned}
E_j(Q_{\eps,b})+k
&\ge \inf_{\substack{S\subset \cD(q_{\eps,b})\\ \dim S=j}} \sup_{u\in S, u\ne 0}
\dfrac{q_{\eps,b}(\chi_1 u,\chi_1 u) + q_{\eps,b}(\chi_2 u,\chi_2 u)}{\|\chi_1 u\|^2_{L^2(V_{\eps,\ell})}+\|\chi_2 u\|^2_{L^2(V_{\eps,\ell})}}\\
&= \inf_{\substack{S\subset \cD(q_{\eps,b})\\ \dim S=j}} \sup_{u\in S, u\ne 0}
\dfrac{t_{\eps,a}(\chi_1 u,\chi_1 u) + r_\eps(\chi_2 u,\chi_2 u)}{\|\chi_1 u\|^2_{L^2(V_{\eps,a})}+\|\chi_2 u\|^2_{L^2(W_\eps)}}\\
&\ge \inf_{\substack{S\subset \cD(t_{\eps,a})\oplus \cD(r_\eps)\\\dim S=j}} \sup_{(u_1,u_2) \in S}
\dfrac{t_{\eps,a}(u_1,u_1) + r_\eps(u_2,u_2)}{\|u_1\|^2_{L^2(V_{\eps,a})}+\|u_2\|^2_{L^2(W_\eps)}}\\
& =E_j\big(T_{\eps,a}  \oplus R_\eps\big).
\end{aligned}
   \label{eq-esta}
\end{equation}

\noindent Let us now obtain a lower bound for $R_\eps$. Using Fubini's theorem, for $u\in H^1(W_\eps)$ one has
\begin{align}
 \int_{W_\eps} |\nabla u|^2 dx  & -  \int_{\partial_0 W_\eps} u^2\, d\sigma=
\int_{a/2}^\ell \bigg[\int_{\B_{\eps x_1^p}} |\nabla u|^2 dx' - \sqrt{1+\eps^2p^2 x_1^{2p-2}} \int_{\partial\B_{\eps x_1^p}} u^2 d\tau  \bigg]dx_1\nonumber\\
&\ge
\int_{a/2}^\ell \bigg[\int_{\B_{\eps x_1^p}} |\nabla_{x'} u|^2 dx' - \sqrt{1+\eps^2p^2 x_1^{2p-2}} \int_{\partial\B_{\eps x_1^p}} u^2 d\tau\bigg]dx_1\nonumber \\
& \ge
\int_{a/2}^\ell \bigg[ E_1 \Big(B_{\eps x_1^p,  \sqrt{1+\eps^2p^2 x_1^{2p-2}}}\Big) \int_{\B_{\eps x_1^p}} u^2 dx'\bigg]\,dx_1\ge \Lambda \int_{W_{\eps}} u^2 dx \label{Lambda}
\end{align}
with
\[
\Lambda:=\inf_{x_1\in (a/2,\ell)} E_1 \Big(B_{\eps x_1^p,  \sqrt{1+\eps^2p^2 x_1^{2p-2}}}\Big).
\]
Due to the parts (a) and (c) of Lemma~\ref{lem-1} one has 
\begin{align*}
& E_1 \Big(B_{\eps x_1^p,  \sqrt{1+\eps^2p^2 x_1^{2p-2}}}\Big)=\dfrac{1}{\eps^2 x_1^{2p}}\ E_1 \Big(B_{1,  \eps x_1^p \sqrt{1+\eps^2p^2 x_1^{2p-2}}}\Big)\\
& \quad =-\dfrac{n\, \sqrt{1+\eps^2p^2 x_1^{2p-2}}}{\eps x_1^{p}}+  (1+\eps^2p^2 x_1^{2p-2}) \, \varphi\big( \eps x_1^p \sqrt{1+\eps^2p^2 x_1^{2p-2}}\,\big)
\end{align*}
with $\varphi\in L^\infty(0,\infty)$. Hence,
\[
E_1 \Big(B_{\eps x_1^p,  \sqrt{1+\eps^2p^2 x_1^{2p-2}}}\Big)=- n\, \sqrt{\dfrac{1}{x_1^{2p}}+\dfrac{\eps^2 p^2}{x_1^2}}\,\cdot \dfrac{1}{\varepsilon}
+
 (1+\eps^2p^2 x_1^{2p-2}) \,\varphi\Big( \eps x_1^p \sqrt{1+\eps^2p^2 x_1^{2p-2}}\,\Big).
\]
Note that under the assumptions $0<\eps<\eps_0$ and $a/2<x_1<\ell$ one has
\begin{align*}
B_1 &:= \sqrt{\dfrac{1}{a^{2p}}+\dfrac{\eps_0^2 p^2}{a^2}}\ge \sqrt{\dfrac{1}{x_1^{2p}}+\dfrac{\eps^2 p^2}{x_1^2}},\\
B_2 & :=  1+b^2 p^2 \equiv 1+\eps^2p^2 \ell^{2p-2}\,\geq 1+\eps^2p^2 x_1^{2p-2},
\end{align*}
implying
\[
E_1 \Big(B_{\eps x_1^p,  \sqrt{1+\eps^2p^2 x_1^{2p-2}}}\Big)\ge -\dfrac{n\, B_1}{\eps} - B_2 \|\varphi\|_\infty.
\]
Hence $\Lambda\ge - nB_1 /\eps - B_2 \|\varphi\|_\infty \ge -B/\eps$
for $B:= n\, B_1+B_2 \|\varphi\|_\infty \eps_0$.
Therefore, by \eqref{Lambda}, for $j\le \cN(T_{\eps,a},-B/\eps)$ one has $E_j\big(T_{\eps,a}  \oplus R_\eps\big)=E_j(T_{\eps,a})$, and the substitution into \eqref{eq-esta} gives the result.
\end{proof}
\begin{proof}[\bf Proof of Proposition \ref{prop-main}]
The upper bound \eqref{main-upperb} follows by combining inequality \eqref{ineq2} with Proposition~\ref{prop7}. On the other hand,  Lemma~\ref{lem11} together with Proposition~\ref{prop10} imply the lower bound \eqref{main-lowerb}. 
\end{proof}

\section*{\bf Acknowledgements}  H.~K.  has been partially supported by Gruppo Nazionale per Analisi Matematica, la Probabilit\`a e le loro Applicazioni (GNAMPA) of the Istituto Nazionale di Alta Matematica (INdAM). The support of  MIUR-PRIN2010-11 grant for the project ``Calcolo delle variazioni'' (H.~K.), is also gratefully acknowledged.

\end{document}